\documentclass[english]{amsart}

\usepackage[english]{babel}
\usepackage{bbm}
\usepackage{hyperref}
\usepackage{verbatim}
\usepackage{ifthen}
\usepackage{amsfonts,amssymb,mathtools}
\usepackage{graphicx,xcolor,subcaption}
\graphicspath{{Figures/}}

\usepackage{pgf}
\newcommand\inputpgf[2]{{
\let\pgfimageWithoutPath\pgfimage
\renewcommand{\pgfimage}[2][]{\pgfimageWithoutPath[##1]{#1/##2}}
\input{#1/#2}
}}

\usepackage[style=numeric,maxcitenames=2,maxbibnames=10,doi=false,isbn=false,url=false,natbib=true,giveninits=true,hyperref,bibencoding=utf8,backend=biber]{biblatex}
\AtEveryBibitem{%
  \clearfield{note}%
  \clearfield{issn}%
  \clearlist{language}%
  }

\usepackage{csquotes}

\newcommand{\N}{\ensuremath{\mathbb{N}}}
\newcommand{\R}{\ensuremath{\mathbb{R}}}

\newcommand{\Z}{\ensuremath{\mathbb{Z}}}
\newcommand{\E}{\ensuremath{\mathbb{E}}}
\renewcommand{\P}{\ensuremath{\mathbb{P}}}

\newcommand{\ind}[1]{\ensuremath{\mathbbm{1}_{\left\{#1\right\}}}}
\newcommand{\diff}{\mathop{}\mathopen{}\mathrm{d}}
\newcommand{\cal}[1]{\ensuremath{\mathcal{#1}}}
\newcommand\croc[1]{\left\langle #1\right\rangle}
\newcommand\steq[1]{\stackrel{\text{\rm #1.}}{=}}

\def\eps{\varepsilon}
\def\cadlag{c\`adl\`ag }
\def\SMS{ Appendix}

\newtheorem{proposition}{Proposition}
\newtheorem{definition}[proposition]{Definition}
\newtheorem{lemma}[proposition]{Lemma}
\newtheorem{theorem}[proposition]{Theorem}

\title{Stochastic Models of Neural Synaptic Plasticity}

\date{\today}

\author[Ph. Robert]{Philippe Robert}
\email{Philippe.Robert@inria.fr}
\urladdr{http://www-rocq.inria.fr/who/Philippe.Robert}
\address[Ph.~Robert, G.~Vignoud]{INRIA Paris, 2 rue Simone Iff, 75589 Paris Cedex 12, France}
\author[G. Vignoud]{Ga\"etan Vignoud${}^1$}
\email{Gaetan.Vignoud@inria.fr}
\address[G. Vignoud]{ Center for Interdisciplinary Research in Biology (CIRB) - Coll\`ege de France (CNRS UMR 7241, INSERM U1050), 11 Place Marcelin Berthelot, 75005 Paris, France}
\thanks{${}^1$Supported by PhD grant of \'Ecole Normale Sup\'erieure, ENS-PSL}

\ifpdf
\hypersetup{
  pdftitle={Stochastic Models of Neural Synaptic Plasticity},
  pdfauthor={Philippe Robert and Ga\"etan Vignoud}
}
\fi

\begin{document}

\begin{abstract}
In neuroscience, learning and memory are usually associated to long-term changes of neuronal connectivity.
In this context, {\em synaptic plasticity} refers to the set of mechanisms driving the dynamics of neuronal connections, called {\em synapses} and represented by a scalar value, the \emph{synaptic weight}.
Spike-Timing Dependent Plasticity (STDP) is a biologically-based model representing the time evolution of the synaptic weight as a functional of the past spiking activity of  adjacent neurons.

If numerous models of neuronal cells have been proposed in the mathematical literature, few of them include a variable for the time-varying strength of the connection.
A new, general, mathematical framework is introduced to study synaptic plasticity associated to different STDP rules.
The system composed of two neurons connected by a single synapse is investigated and a stochastic process describing its dynamical behavior is presented and analyzed.
The notion of {\em plasticity kernel} is introduced as a key component of plastic neural networks models, generalizing a notion used for pair-based models.
We show that a large number of STDP rules from neuroscience and physics can be represented by this formalism.
Several aspects of these models are discussed and compared to canonical models of computational neuroscience.
An important sub-class of plasticity kernels with a Markovian formulation is also defined and investigated.
In these models, the time evolution of cellular processes such as the neuronal membrane potential and the concentrations of chemical components created/suppressed by spiking activity  has the Markov property.
\end{abstract}

\maketitle

\vspace{-5mm}

\bigskip

\hrule

\vspace{-3mm}

\tableofcontents

\vspace{-1cm}

\hrule

\section{Introduction}
\label{sec:Introduction}
Central nervous systems, as the brain, are the main substrate for memory and learning, two essential concepts in the understanding of behavior.

It is widely accepted that neurons constitute the main relay for information in complex neural networks composing the brain.
This multi-scale system, ranging from single neuronal cells to complex brain areas, is known to be the basis of memory consolidation, i.e. the transformation of a temporary information into a long-lasting stable memory.
The memory trace, or \emph{engram}, is the focus of studies in neuroscience, see~\citet{tonegawa_role_2018} for example.
Biological, computational and mathematical models are developed to understand mechanisms by which an \emph{engram} emerges during learning, maintains itself, and evolves with time.

{\em Synapses} are the key components for the transmission of information between connected neurons, and accordingly,
it is assumed that the encoding of memory is integrated in the intensity of these connections.
From a biological point of view, a synapse is a structure, located at the junction of two neurons, where the transmission of chemical/electrical signals is possible.
A neuronal connection is unidirectional in the sense that the signal goes from an input neuron, called the {\em pre-synaptic neuron}, to the output one, {\em the post-synaptic neuron}.
The intensity of the connection is referred to as the {\em synaptic efficacy/strength} and is represented by a scalar variable, the {\em synaptic weight} $W$.
The impact of an input signal, a {\em spike}, from the pre-synaptic neuron is modeled as a jump of the {\em membrane potential} $X$ of the post-synaptic neuron.
The amplitude of this jump is used to quantify the synaptic weight.

A synaptic plasticity mechanism is defined as a collection of activity-dependent cellular processes that modifies the synaptic connectivity.
During learning, specific patterns of neural activity may elicit short, from milliseconds to seconds, and/or long, from minutes to hours, -term changes in the associated synaptic weights. In this context, it is conjectured that memory is directly associated to synaptic plasticity, see~\citet{takeuchi_synaptic_2014}.

\subsection{The State of a  Neuronal Cell}
\label{subsec:stateneuralcell}

In this paper we will investigate stochastic models of the dynamic of the synaptic weight of a connection from a pre-synaptic neuron to a post-synaptic neuron. 

The post-synaptic neuron is represented by its membrane potential $X$ which is a key parameter to describe its current activity. In neuroscience numerous models of an individual neuronal cell and neuronal networks have been used to investigate learning abilities and plasticity.
See~\citet{gerstner_neuronal_2014} for a review.

The {\em leaky-integrate-and-fire model}  describes the time evolution of the membrane potential as a resistor-capacitor circuit with a constant leaking mechanism. Due to different input currents, the membrane potential of a neuron may rise until it reaches some threshold after which a spike is emitted and transferred to the synapses of neighboring cells. A large class of neural models based on this hypothesis has been developed, see~\citet{gerstner_neuronal_2014} and references within.

To take into account the important fluctuations within cells, due to the spiking activity and thermal noise in particular, a random component in the cell dynamics has to be included in mathematical models describing the membrane potential evolution. For several models this  random component is represented as an independent additive diffusion component, like Brownian motion, of the membrane potential.

In our approach, the random component is  at the level of the generation of spikes.
When the value of the  membrane potential of the output neuron is at $X{=}x$, a spike occurs at rate $\beta(x)$ where $\beta$ is the {\em activation function}. See~\citet{chichilnisky_simple_2001} for a discussion.
In particular the instants when the output neuron spikes are represented by an inhomogeneous Poisson process.
Considering a constant synaptic weight $W$, the time evolution of the post-synaptic membrane potential $(X(t))$ is represented by the following stochastic differential equation (SDE):
\begin{equation}\label{IX0}
  \diff X(t) = {-}\frac{1}{\tau}X(t)\diff t+W\mathcal{N}_{\lambda}(\diff t)-g\left(X(t{-})\right)\mathcal{N}_{\beta, X}(\diff t),
\end{equation}
where $X(t{-})$ is the left limit of $X$ at $t{>}0$,
\begin{itemize}
\item  $\tau$ is the exponential decay time constant of the membrane potential associated to the leaking mechanism;
\item The sequence of firing instants of the pre-synaptic neuron  is represented by a Poisson point process $\mathcal{N}_{\lambda}$ on $\R_+$ with rate $\lambda$. At each pre-synaptic spike, the membrane potential $X$ is increased by the amount $W$;
\item The sequence of firing instants of the post-synaptic neuron is an inhomogeneous Poisson point process $\mathcal{N}_{\beta, X}$ on $\R_+$ whose rate function  is given by $(\beta(X(t{-})))$.
\item The drop of potential due to a post-synaptic spike is represented by the function $g$, i.e.
after a post-synaptic spike, the membrane potential is reset to $X(t-){-}g(X(t-))$.
\end{itemize}
Considering that the point process $\mathcal{N}_{\beta, X}$ depends on $(X(t))$, Relation~\eqref{IX0} can be seen as a fixed point equation.

\subsection{Synaptic plasticity}
\label{subsec:synapticplasticity}
Synaptic plasticity refers to different mechanisms that leads to the modification of the synaptic weight.
Consequently, we need to consider a time varying version of the synaptic strength $W(t)$.
Although synaptic plasticity is a complex mechanism, general principles have been inferred from experimental data and previous modeling studies.
One of the founding principles is Hebb's postulate~(1949), later on summarized by~\citet{shatz_developing_1992} as, ``{\em Cells that fire together wire together}''.

Synaptic {\em potentiation}, resp. {\em depression}, is associated to an increase, resp. a decrease, of the synaptic strength.
Plasticity is described as a set of mechanisms controlling the potentiation and the depression of synapses.
It usually depends on the pre-synaptic and post-synaptic signaling, i.e.~of past instants of pre-synaptic and post-synaptic spikes.
In the literature this class of mechanisms are referred to as {\em Spike-Timing Dependent Plasticity} (STDP).
Several experimental protocols have been developed to elicit STDP at synapses: sequences of spikes pairing from either side of a specific synapse are presented, at a certain frequency and with a certain delay.
Occurrence, magnitude and polarity of STDP have been shown to depend on protocols used in experiments: frequency, number of pairings, types of synapses where it is applied, the neuronal sub-population, brain area, just to cite a few key parameters, see~\citet{feldman_spike-timing_2012}.

We now introduce two important classes of synaptic plasticity mechanisms.
Most models of the literature belong to, or are a variation of, one of these two classes.
\begin{enumerate}
\item {\textsc{Pair-based models.} }\\
  Each pair $t{=}(t_{\rm pre},t_{\rm post})$ of instants of pre-synaptic and post-synaptic spikes is associated to an increment $\Delta W$ of the synaptic weight at time $\max(t_{\rm pre},t_{\rm post})$,
\begin{equation}\label{IDeltaWP}
\Delta W = \Phi(\Delta t),
\end{equation}
where $\Delta t{\steq{def}}t_{\rm post}{-}t_{\rm pre}$ and $\Phi$ is the some function on $\R$, the {\em STDP curve}.
The function $\Phi$, usually taken from experimental data, is sharply decreasing to $0$ as $\Delta t$ goes to infinity, so that distant spikes have a negligible contribution.
\noindent

Many variants and extensions of pair-based models have been developed over the years to fit with experimental results.
Triplets-rules, described in Sections~\ref{secsecsec:STDPSup} and ~\ref{secsecsec:Triplet}, add a dependency between spikes of the same neuronal cell.
Additional examples can be found in~\citet{babadi_stability_2016}.

\medskip

\item {\textsc{Calcium-based models.} }\\
Another class of models infers from explicit biological mechanisms the shape of the STDP curve.
Post-synaptic calcium traces have been found experimentally to be critical in the establishment of plasticity, see~\citet{feldman_spike-timing_2012} and references therein.
In a classical model, when the calcium concentration $C_{\rm ca}$ in the post-synaptic neuron reaches some specific threshold, STDP is induced accordingly.
The analogue of Relation~\eqref{IDeltaWP} for calcium-based STDP rules is,
\begin{equation}\label{IDeltaWC}
\diff W(t)  = F(C_{\rm ca}(t))\diff t,
\end{equation}
for some function $F$.
The dynamics of $C_{\rm ca}$ is only driven by instants of pre- and post-synaptic spikes.
Consequently, the dependence of plasticity on the instants of spikes is not expressed directly as in pair-based models, but through some intermediate biological variable.
Several biophysical models are based on this calcium hypothesis, see~\citet{graupner_mechanisms_2010} for a review.
\end{enumerate}

It should be noted that there are other  STDP models, such as the ones based on exponential filtered traces of the membrane potential, see~\citet{clopath_voltage_2010}. Pair-based and calcium-based models are nevertheless the most widely used STDP rules in  large-scale plastic neural networks.

\subsection{Models of Plasticity in the Literature}
\label{subsec:modelsofplasticityliterature}
To understand how synaptic plasticity may shape the brain, the study of STDP in neural networks has attracted a lot of interest in different domains:
\begin{enumerate}
\item {\em Experiments},  with measurements of a large variety of STDP rules;
\item {\em Computational models}, for numerical simulations of these protocols with several populations of neuronal cells;
\item {\em Mathematical models}, to investigate the qualitative properties of STDP rules.
\end{enumerate}

Many computational models have been developed to investigate STDP rules in different contexts.
See~\citet{kempter_hebbian_1999} and~\citet{morrison_spike-timing-dependent_2007} and the references therein.

Mathematical studies of models of plasticity are quite scarce. Most models are centered on evolution equations of neural networks with a fixed synaptic weight. See Sections~1 and~2 of~\citet{robert_dynamics_2016} for a review.  \citet{helson_new_2018} investigates  a Markovian model of a  Nearest Neighbor Symmetric Model  STDP rule. See Section~\ref{secsecsec:pairbased}. This is one of the few stochastic analyses in this domain.

\subsection{Contributions}
\label{subsec:contributions}
A mathematical model of plasticity describing a pre- and a post-synaptic neuron should include the spiking mechanisms of the two neuronal cells. It is given by the time evolution of the membrane potential~$X$ of the post-synaptic cell, as described by Equation~\eqref{IX0}. It must also include the dynamics of plasticity of the type~\eqref{IDeltaWP} or~\eqref{IDeltaWC} for the time evolution of the synaptic weight $W$. 

The difficulty lies in the complex dependence of the evolution of $W$  with respect to the instants of spikes of both cells, the processes ${\cal N}_\lambda$ and ${\cal N}_{\beta,X}$ of Equation~\eqref{IX0}. For pair-based models for example, this is a functional of all pairs of instants of both processes. In general, there does not exist a simple Markovian model to describe the membrane potential dynamics and the evolution of the synaptic weight. 

In Section~\ref{sec:modelneuralplasticicty}, we introduce the notion of {\em plasticity kernel} which describes in a general way how the spiking activity is taken into account in the synaptic weight dynamics as a functional of the point processes ${\cal N}_\lambda$ and ${\cal N}_{\beta,X}$.
A differential system associated to the dynamics of the variables~$X$ and $W$ is presented.
Under mild conditions, it is proved that it has a unique solution for a given initial state.
It is, to the best of our knowledge, the first attempt to have a general mathematical framework that describes most STDP rules of the literature. A large set of examples is presented in Section~\ref{secsec:plasticitykernel} and Section~\ref{AsecMod}: most STDP models of~\citet{morrison_phenomenological_2008,graupner_mechanisms_2010,clopath_connectivity_2010,babadi_stability_2016} can be represented within this formalism. Section~\ref{ap:figures}  gives a graphical representation of several STDP rules, see Figures~\ref{figS:pairbased} and~\ref{figS:calciumbased}.

Section~\ref{sec:markov} is devoted to an important sub-class of STDP rules, {\em plasticity kernels of class ${\cal M}$}.  These kernels have a representation in terms of a finite dimensional process whose coordinates can be interpreted as concentrations of chemical components created/suppressed by spiking activity. If a classical Markovian analysis of the associated stochastic processes is not really possible, their main advantage is that one can formulate a tractable model with two timescales, when the cellular dynamics are ``fast''. This approach is developed in the follow-up paper~\citet{robert_stochastic_2020_2}. For these models, when the synaptic weight is fixed, the fast stochastic processes have the Markov property. Section~\ref{sec:Cell} discusses these aspects and several examples are presented in Section~\ref{ap:generator}.
Finally in Section~\ref{ap:stochqueu}, a discrete formulation of the stochastic system is defined and its fast processes invariant distribution is analyzed. The case of a calcium-based model is analyzed. 
Section~\ref{apap:noexp}  discusses modeling issues on the incorporation of plasticity:  via a time-smoothing kernel, as we do in the paper, or directly with an instantaneous information.

\subsection{STDP in recurrent neural networks}
In this paper, we consider only two neurons (the pre-synaptic neuron and the post-synaptic) that are connected by a single synapse.
As it will be seen, a large variety of models have been used in the literature to describe the time evolution of a synaptic weight.
Our goal is to propose a general, basic, mathematical framework where most of these models of plasticity can be investigated.
The dynamics of the synaptic weight $(W(t))$ depends in an intricate way on the point process ${\cal N}_\lambda$ for pre-synaptic spikes and ${\cal N}_{\beta,X}$ for post-synaptic spikes.

For a neural network whose nodes are the vertices of a graph ${\cal G}$, an extension of this model would be as follows: the membrane potential process $(X_i(t))$ of node $i{\in}{\cal G}$ should satisfy the SDE,
\[
\diff X_i(t) {=} \displaystyle {-}\frac{1}{\tau}X_i(t)\diff t{+}\sum_{j{\to}i}W_{j, i}(t{-})\mathcal{N}_{\beta, X_j}(\diff t){-}g(X_i(t-))\mathcal{N}_{\beta,X_i}(\diff t),
\]
where $j{\to}i$ indicates that there is a synapse  $(j,i)$, from node $j$ to  node $i$, and $(W_{j,i}(t))$ is the corresponding process for the synaptic weight. The associated differential quantity $W(t{-}){\cal N}_\lambda(\diff t)$ for  instants of pre-synaptic of the synapse $(j,i)$ is  given by
\[
W_{j, i}(t{-})\mathcal{N}_{\beta, X_j}(\diff t).
\]
Each synaptic weight $(W_{j, i}(t))$ will be subject to synaptic plasticity, with defined plasticity kernels $\Gamma_{p,i}$ and $\Gamma_{d,i}$ that can be different.
For synaptic weight $(W_{j, i}(t))$, we will define ${\cal N}_{\beta,X_j}$ as the Poisson process representing the pre-synaptic neuron and similarly, ${\cal N}_{\beta,X_i}$ for the post-synaptic neuron.

The models and some results of our paper can be extended to the multidimensional case, in particular the existence and uniqueness result,  Theorem~\ref{th:ExistTheo}.
For simplicity and because of its importance as a generic model, we will restrict ourselves to the case of a network with two-nodes.

\section{Models of Neural Plasticity}
\label{sec:modelneuralplasticicty}
We consider two neurons connected by one {\em synapse}. A synapse is a unidirectional connection from the {\em input neuron} to the {\em output neuron} allowing the transmission of `information'. When the input, or {\em pre-synaptic},  neuron  spikes, some neurotransmitters are released at the level of the synapse, where they can interact with the output, or {\em post-synaptic}, neuron.
Following synaptic transmission, a pre-synaptic spike increments the {\em membrane potential} $X$ of the output neuron by a scalar value, the synaptic weight~$W$.

The dynamics of neural plasticity is described in terms of the time evolution of $(X(t))$ and $(W(t))$.
For $t{\ge}0$,
\begin{enumerate}
\item $X(t){\in}\R$ is the {\em membrane potential} of the output neuron at time $t$.
 This is the difference between the internal and the external electric potentials of the neuron.
 The dynamics of the process $(X(t))$ associated to the output neuron is a classical model of neuroscience. See~\citet{gerstner_neuronal_2014} for a survey.
\item $W(t){\in}\R$ represents the intensity of {\em synaptic transmission} at time $t$, i.e. the increment of the post-synaptic membrane potential $X$ when the input neuron spikes at time $t$.
The evolution of $(W(t))$ at time $t{>}0$ depends in general on the total sample path of $((X(s),W(s)), 0{\leq}s{\leq}t)$, in an intricate way.
\end{enumerate}
To take into account inhibitory mechanisms, these two variables are real-valued and, consequently, may have negative values.
Real synapses have a constant sign: they can be either excitatory (with a non-negative synaptic weight) or inhibitory (with a non-positive synaptic weight).
In the following sections, other variables will be added to formalize the evolution equations of $(X(t),W(t))$.

\subsection{Definitions and Notations}
\label{secsec:definitionnotations}
Sequences of pre- and post-synaptic spikes play an important role in the study of {\em spike-timing dependent} plasticity.
Mathematically, it is convenient to describe them in terms of {\em point processes}.
See~\citet{dawson_measure-valued_1993} for  general definitions and results on point processes.

We denote by ${\cal M}_+(\R_+^d)$  the set  of  positive Radon measures on $\R_+^d$, i.e.~with finite values on any compact subset of $\R_+^d$.
A point measure on $\R_+^d$, $d{\ge}1$, is an integer-valued Borelian positive measure on $\R_+^d$ which is Radon.
A point measure is carried by a subset of $\R_+^d$ which is at most countable and without any finite limiting point.
The set of point measures on $\R_+^d$ is denoted by ${\cal M}_p(\R_+^d){\subset}{\cal M}_+(\R_+^d)$, it is endowed with the natural weak topology of ${\cal M}_+(\R_+^d)$ and its corresponding Borelian $\sigma${-}field.

If $m{\in}{\cal M}_p(\R_+^d)$ and $A{\in}{\cal B}(\R_+^d)$ is a Borelian subset of $\R_+^d$, then $m(A)$ denotes the number of points of $m$ in $A$, i.e.
\[
m(A)=\int_{\R_+^d}\mathbbm{1}_{A}(x)\,m(\diff x).
\]
A point process on $\R_+^d$ is a probability distribution on ${\cal M}_p(\R_+^d)$.
Two independent Poisson point processes are assumed to be defined on a filtered probability space $(\Omega,{\cal F},({\cal F}_t), \P)$, see~\citet{kingman_poisson_1992},
\begin{enumerate}
\item A point process ${\cal N}_\lambda$ on $\R_+$ to represent the instants of pre-synaptic spikes is assumed to be Poisson  with rate $\lambda{>}0$, $(t_{{\rm pre},n})$ is the increasing sequence of its jumps, i.e.\ \[
{\cal N}_\lambda{=}\sum_{n{\ge}1} \delta_{t_{{\rm pre},n}}, \text{ with  } 0{\le}t_{{\rm pre},1}{\le}t_{{\rm pre},2}\le\cdots \le t_{{\rm pre},n}\le\cdots,
\]
where $\delta_a$ is the Dirac measure at $a{\in}\R_+$;
\item  A Poisson point process ${\cal P}$ on $\R_+^2$ with rate  $1$. It is used to define the inhomogeneous point process of post-synaptic spikes in Relation~\eqref{Nbeta}. 
\end{enumerate}

The variable $t$ of the point processes ${\cal N}_\lambda(\diff t)$ and  ${\cal P}(\diff x,\diff t)$ is interpreted as the time variable.
For $t{\ge}0$, the $\sigma$-field ${\cal F}_t$ of the filtration $({\cal F}_t)$ of  the probability space is assumed to contain all events before time $t$ for both point processes, i.e.\
\begin{equation}\label{eq:ft}
  \sigma\left< \rule{0mm}{4mm}\mathcal{P}_1\left(\rule{0mm}{3mm}A{\times}(s,t]\right), \mathcal{P}_2\left(\rule{0mm}{3mm} A{\times}(s,t]\right), A {\in}\mathcal{B}\left( \mathbb{R}_+ \right), s{\leq}t \right>\subset {\cal F}_t.
\end{equation}

A stochastic process $(U(t))$ is {\em adapted} if, for all $t{\ge}0$,  $U(t)$ is ${\cal F}_t$-measurable.
It is a {\em  \cadlag process} if, almost surely, it is right continuous and has a left limit at every point $t{>}0$, $U(t{-})$ denotes the left limit of $(U(t))$ at $t$. The Skorokhod space of \cadlag functions from $[0,T]$ to $\mathcal{S}$ is $\mathcal{D}([0,T],\mathcal{S})$. See~\citet{billingsley_convergence_1999}.

The set of real continuous bounded functions on the metric space $\mathcal{S}{\subset}\R^d$  is denoted by $\mathcal{C}_b(\mathcal{S})$, and $\mathcal{C}_b^k(\mathcal{S}){\subset}\mathcal{C}_b(\mathcal{S})$ is the set of bounded, $k$-differentiable functions on $\mathcal{S}$ with respect to each coordinate, with the respective derivatives bounded and continuous.

We conclude this preliminary section with an elementary but important lemma concerning the {\em filtering} of a stochastic process with an exponential function.
\begin{lemma}[Exponential Filtering]\label{FiltLem}
  If $\mu$ is a non-negative Radon measure on $\R_+$, $\alpha{>}0$ and $h_0{\in}\R$, then
  \[
  H(t)=h_0 e^{-\alpha t}{+}\int_{(0,t]} e^{-\alpha(t-s)}\mu(\diff s)
    \]
    is the unique \cadlag solution  of the differential equation,
\[    \diff H(t){=} {-}\alpha H(t)\diff t {+}\mu(\diff t),
\]
   such that $H(0){=}h_0$. 

\end{lemma}
This type of process is a central object in mathematical models of neuroscience.
It is used to represent {\em leaky-integrate phenomena} of chemical components within cells.
See~~\citet{gerstner_neuronal_2014} for a general review.

\subsection{The Dynamics of the Post-synaptic Membrane Potential}
\label{secsec:TimeEvolMP}

It is represented as a \cadlag stochastic process $(X(t))$ following leaky-integrate dynamics illustrated in Figure~\ref{figsub:model}:
\begin{enumerate}
\item It decays exponentially to $0$ with a fixed characteristic decay time $\tau$, set without loss of generality to $\tau{=}1$.
\item It is incremented by the current synaptic weight variable  at each firing instant of the input neuron, i.e. at each instant of the Poisson point process~${\cal N}_\lambda$.
\item The firing mechanism of the output neuron is driven by a function $\beta$ from $\R$ to $\R_+$, the activation function.
When the membrane potential is $x$, the output neuron fires at rate $\beta(x)$.
This function is usually assumed to be non-decreasing, in other words, the larger the membrane potential is, the more likely the neuron is to spike.

\item After a post-synaptic spike, the neuronal membrane potential $X$ is decreased by the amount $g
(x)$, where $g$ is some function on
$\R$.
In general, the membrane potential is reset to $0$ after a spike, i.e. $g(x){=}x$,  see~\citet{robert_dynamics_2016}. However,  in some cases,  the reset potential may not depend on the membrane potential before the spike, $g$ can be  constant for example.
\end{enumerate}

\noindent
{\sc Post-synaptic spikes}.
If the instants of pre-synaptic spikes are represented by the Poisson process ${\cal N}_\lambda$,
the firing instants of the output neuron $t_{post,n}$ are expressed as the jumps of the point process ${\cal N}_{\beta,X}$ on $\R_+$ defined by
\begin{multline}\label{Nbeta}
  \int_{\R_+}f(u){\cal N}_{\beta,X}(\diff u)\steq{def}\int_{\R_+}f(u){\cal P}\left(\rule{0mm}{4mm}\left(\rule{0mm}{3mm}0,\beta(X(u-))\right],\diff u\right)\\
=  \int_{\R_+^2}f(u)\ind{s{\in}(0,\beta(X(u-))]}{\cal P}(\diff s,\diff u),
\end{multline}
for any non-negative Borelian function $f$ on $\R_+$.

Classical properties of Poisson processes give that, for $t{>}0$ and  $x{\in}\R$,
\[
\P\left.\left(\rule{0mm}{5mm}{\cal N}_{\beta,X}(t,t{+}\diff t){\neq} 0 \right| X(t-){=}x\right){=}\beta(x)\diff t{+}o(\diff t),
\]
as expected, ${\cal N}_{\beta,X}$ is Poisson process with intensity $(\beta(X(t)))$.

\medskip
The following stochastic differential equation summarizes the description of the time evolution of $(X(t))$  given by a), b), c) and~d),
\[
  \diff X(t) = -X(t)\diff t+W(t{-})\mathcal{N}_\lambda(\diff t)-g(X(t-))\mathcal{N}_{\beta,X}(\diff t).
\]

\subsection{Time Evolution of the  Synaptic Weight}
\label{secsec:timeevolutionweight}

In this work, the synaptic weight $W$ will stay in a defined real (not necessarily bounded) interval $K_W$.
For several examples, the plasticity process leads to  dynamics for which the process $(W(t))$ stays in $K_W$ for all time $t$.
\begin{itemize}
\item Taking $K_W{=}\R$ leads to free dynamics of the synaptic weight, that can be either negative or positive, change
 its sign because of the plasticity rules. This situation occurs in models of neural networks  where  excitatory/inhibitory neurons are not separated in distinct classes..
 \item If $K_W{=}\R_+$, the synaptic weight is non-negative and plasticity processes cannot change its sign. This is a model for excitatory
 neurons whose spikes lead to the increase of the post-synaptic membrane potential.
 \item Conversely, if $K_W{=}\R_-$, the cell is an inhibitory neuron, which has the opposite effect on the  post-synaptic membrane potential.
 \item Finally, $K_W$ can also be bounded in order to represent saturation mechanisms, i.e. the synaptic weights needs to stay in a biological range of value.
   In that case, potentiation refers to a diminution of the amplitude of the negative jump, whereas depression indicates an augmentation. In experimental works, the denominations are inverted, for the sake of clarity we chose to stay with the previous names.
\end{itemize}

We can now introduce the notion of  {\em plasticity kernels}.
\begin{definition}[Plasticity Kernel]
\label{def:gamma}
A plasticity kernel is a measurable function
\[
    \Gamma{:}\ {\cal M}_p(\R_+)^2\longrightarrow {\cal M}_+(\R_+), \quad
    (m_1,m_2) \longrightarrow \Gamma(m_1,m_2),
\]
${\cal M}_+(\R_+)$ is  the set  of  positive Radon measures on $\R_+$ and, for any $t{>}0$, the functional
\begin{equation}\label{Consist}
 (m_1,m_2)\longrightarrow  \Gamma(m_1,m_2)(\diff u \cap[0,t])
\end{equation}
is ${\cal G}_t{\otimes}{\cal G}_t$-measurable,  where $\mu(\diff u{\cap}[0,t])$ denotes the restriction of the Radon measure $\mu$ to the interval $[0,t]$ and $({\cal G}_t)$ is the filtration  on ${\cal M}_p(\R_+)$, such that for $t{\ge}0$,
${\cal G}_t$ is the $\sigma$-field generated by the functionals $m{\to}m((0,s])$, with $s{\le}t$.
\end{definition}
If $\Gamma$ is a plasticity kernel and $m_1$, $m_2{\in}{\cal M}_p(\R_+)$, the measure $\Gamma(m_1,m_2)(\diff u{\cap}[0,t])$ depends only on the variables  $m_i([0,s])$, for $i{\in}\{1,2\}$ and $s{\le}t$.

In our model, the infinitesimal elements at time $t$ for the update of plasticity are expressed as
$\Gamma({\cal N}_\lambda,{\cal N}_{\beta,X})(\diff t)$ for some plasticity kernel $\Gamma$.
This  quantifies how the interaction between the instants of pre-synaptic and of post-synaptic spikes, ${\cal
N}_\lambda$ and ${\cal N}_{\beta,X}$  leads to specific synaptic changes.
For example, the order and timing between instants of pre- and post-synaptic spikes may have an impact on plasticity.

In previous works~\citet{10.3389/fncom.2010.00019,Fremaux13326,10.1371/journal.pone.0101109,feldman_spike-timing_2012}, the notion of STDP Temporal Kernels referred to the curve of synaptic weight change $\Delta W$ as a
function of $\Delta t$ for pair-based models. \citet{pfister_triplets_2006} introduced more complex kernels, with multi-spikes interactions.
The plasticity kernels defined above extend this notion to more general interactions between pre- and post-synaptic spikes.

Plasticity is represented as a process, integrating, with some decay, the past interactions of the spiking activity on either side of the synapse.
Two non-negative process are introduced:  $(\Omega_p(t))$ and $(\Omega_d(t))$,  the first one is associated to
potentiation (increase of $W$) and the other  to depression (decrease of $W$). For $a{\in}\{p,d\}$,
\begin{equation} \label{eq:Omega}
\Omega_a(t) = \Omega_a(0) e^{-\alpha t}+\int_{(0,t]} e^{-\alpha(t-s)} \Gamma_a({\cal N}_\lambda,{\cal N}_{\beta,X})(\diff s),
\end{equation}
where $\alpha{>}0$ and  the variables $\Gamma_p$ and  $\Gamma_d$ are plasticity kernels associated to potentiation and depression respectively.  The process $(\Omega_a(t))$ can be seen as a exponential filtering of the random measure $\Gamma_a({\cal N}_\lambda,{\cal N}_{\beta,X})(\diff t)$ in the sense of Lemma~\ref{FiltLem}.
In Section~\ref{apap:noexp} of \SMS, another stochastic model of plasticity  with no exponential filtering of the plasticity kernels is introduced and discussed.

As explained in the introduction of this section, the function $M$ need to be chosen so that the synaptic weight $W$
stays at all time in its definition interval $K_W$.
The time evolution of $(W(t))$ depends then on the past activity of the input and output neurons, through $(\Omega_p(t))$ and $(\Omega_d(t))$ and is described by,
\begin{equation} \label{eq:M0}
\frac{\diff W(t)}{\diff t} = M\left(\Omega_p(t),\Omega_d(t),W(t)\right),
\end{equation}
with, $M$ verifying, for any piecewise-continuous \cadlag functions $(\omega_p(t))$ and $(\omega_d(t))$ on $\R_+$, a solution $(w(t))$ of the ODE
\[
    \frac{\diff w(t)}{\diff t}=M(\omega_p(t),\omega_d(t),w(t)),
\]
with $w(0){\in}K_W$, is such that $w(t){\in}K_W$,  for all $t{\geq}0$.

We now give some examples of functions $M$ associated to different synaptic domains $K_W$.
For $K_W{=}\R$, we can chose the additive implementation of STDP rules, where,
\begin{equation} \label{eq:M1}
  M(\omega_p,\omega_d,w)\steq{def}M(\omega_p, \omega_d)=\omega_p-\omega_d
\end{equation}
In that case, the dynamics are unbounded and we see the update only depends on the potentiation/depression plasticity
variables $\Omega_a$.

If we want to model bounded synaptic weight in $K_W{=}[A_d,A_p]$, we can consider the function $M$ given by
\begin{equation} \label{eq:M2}
 M(\omega_p,\omega_d,w)\steq{def}
 (A_p{-}w)^{n} \omega_p{-}(w{-}A_d)^{n} \omega_d - \mu (w-A_r), \quad w{\in}[A_d,A_p],
\end{equation}
where  $A_d{\le}A_r{\le}{A_p}$, and $n{>}0$.
This corresponds to a multiplicative influence of
$W$. See~\citet{gutig_learning_2003}.
It is straightforward to see that in that case, the synaptic weight stays bounded between $A_p$ and $A_d$ for any
plasticity processes $\Omega_a$.
The expression ${-}\mu (W(t){-}A_r)$ is for the exponential decay of the synaptic weight $W$ to $A_r$, its resting
value.
This term represents {\em homeostatic mechanisms}, i.e.\ mechanisms that maintain steady internal physical and
chemical conditions to allow the functioning of the system. See~\citet{turrigiano_homeostatic_2004}.

Finally, an unbounded dynamics for an excitatory synapse, with $K_W{=}\R_+$ can be enforced by,
\begin{equation} \label{eq:M3}
 M(\omega_p,\omega_d,w)\steq{def}
 \omega_p{-}w\omega_d,
\end{equation}

\subsection{Examples of Plasticity Kernels}\label{secsec:plasticitykernel}
We show that several important STDP rules of the literature can be
expressed with plasticity kernels $\Gamma_p$ and $\Gamma_d$. Further extensions are presented in~\ref{secsecsec:STDPSup} and Section~\ref{secsecsec:Triplet}
\subsubsection{\sc Pair-Based Models}
\label{secsecsec:pairbased}

For pair-based mechanisms, the synaptic weight is modulated according to the respective timing of pre-synaptic and
post-synaptic spikes, as illustrated in Figure~\ref{figsub:pair}.
This follows the fact that most STDP experimental studies are based on pairing protocols, where pre- and
post-synaptic spikes are repeated at a certain frequency for a given number of pairings.

Accordingly, a large class of models have been developed on the principle that the synaptic weight change due to a pair $(t_{\rm pre},t_{\rm post})$ of instants of pre- and post-synaptic spikes,  depends only  on $\Delta t{=}t_{\rm post}{-}t_{\rm pre}$.
The synaptic update is then taken proportional to $\Phi(\Delta t)$, where $\Phi$ is some function converging to $0$ at infinity, that is referred to as the {\em STDP curve}.
An example of exponential STDP curves is given in Figure~\ref{figsub:pair}~(top left).
Many pair-based models have been developed over the years, varying mainly which pairs of spikes are taken into account when updating the synaptic weight.

We  start with the simplest rule, the \emph{all-to-all} version (following~\citet{morrison_phenomenological_2008}
terminology), where all pairs of spikes give an update of the synaptic weight.

\subsection*{All-to-all Model}
The \emph{all-to-all} scheme consists in updating the synaptic weight at each post-synaptic spike, occurring at time $t$ by the sum over all previous pre-synaptic spikes occurring at time $s{<}t$ of the quantity $\Phi(t{-}s)$.
Switching the role of pre- and post-synaptic spikes, the synaptic weight is updated  in the same way with other constants.
See Figure~\ref{figsub:pair}~(bottom left) for an example of \emph{all-to-all} interactions.

The plasticity kernels are defined by, for $m_1$, $m_2{\in}{\cal M}_p(\R_+)$ and $a{\in}\{p,d\}$,
\begin{multline}    \label{eq:FPa}
\Gamma_{a}^{\textup{PA}}(m_1,m_2)(\diff t) \steq{def}  \left(\int_{(0,t)}\Phi_{a,2}(t{-}s)m_2(\diff s)\right)\,m_1(\diff t)\\{+}
  \left(\int_{(0,t)}\Phi_{a,1}(t{-}s)m_1(\diff s)\right)\,m_2(\diff t), \quad a{\in}\{p,d\}.
\end{multline}
The functions $\Phi_{a,i}$, $a{\in}\{p,d\}$ and $i{\in}\{1,2\}$ are non-negative and non-increasing functions functions converging to $0$ at infinity.

If $f$ is a non-negative Borelian function on $\R_+$, we  have
\begin{multline*}
  \int_{\R_+} f(t)\Gamma_a^{\textup{PA}}({\cal N}_\lambda,{\cal N}_{\beta,X})(\diff t)\\ =\sum_{t_{\text{pre}}}f(u)\sum_{t_{\text{post}}{<}t_{\text{pre}}} \Phi_{a,2}(t_{\text{post}}{-}t_{\text{pre}})
       {+}\sum_{t_{\text{post}}} f(u)\sum_{t_{\text{pre}}<t_{\text{post}}} \Phi_{a,1}(t_{\text{pre}}{-}t_{\text{post}}).
\end{multline*}
\bigskip

\noindent
{\bf Remarks.}
\begin{enumerate}
\item  The  exponential STDP functions $\Phi(s){=}B\exp({-}\gamma s)$, $s{\ge}0$, are often used in this context. See~\citet{morrison_phenomenological_2008}.  Several studies also consider the case when $\Phi$ is a translated exponential kernel. See~\citet{lubenov_decoupling_2008}.

\item {\em Hebbian STDP} plasticity is said to occur when
\begin{itemize}
\item a  pre-post pairing, i.e.~ $t_{\rm pre}{<}t_{\rm post}$  leads to potentiation,  $\Delta W{>}0$;
\item  a post-pre pairing, $t_{\rm post}{<}t_{\rm pre}$,  leads to depression,  $\Delta W{<}0$.
\end{itemize}
Experiments have shown that this type of plasticity occurs for several populations of neuronal cells~\citet{bi_synaptic_1998}.
Early models can be found in~\citet{van_rossum_stable_2000,rubin_equilibrium_2001,morrison_phenomenological_2008} for a review.

Following Hebb's postulate, a `causal' pre-post pairing (a post-synaptic spike occurs after a pre-synaptic one) should lead to potentiation,
\[
\Gamma_{p}^{\textup{PAH}}(m_1,m_2)(\diff t) = \left(\int_{(0,t)}\Phi_{p,1}(t{-}s)m_1(\diff s)\right)\,m_2(\diff t).
\]
Conversely, a \emph{post-pre} pairing (anti-causal activation) leads to depression,
\[
\Gamma_{d}^{\textup{PAH}}(m_1,m_2)(\diff t) = \left(\int_{(0,t)}\Phi_{d,2}(t{-}s)m_2(\diff s)\right)\,m_1(\diff t).
\]
This corresponds to $\Phi_{p,2}{=}0$ and $\Phi_{d,1}{=}0$ in Equation~\eqref{eq:FPa}.

\item Other forms of STDP have been discovered experimentally see~\citet{feldman_spike-timing_2012}. {\em Anti-Hebbian STDP} models follows the opposite principles:  Pre-post pairings lead to depression, and  post-pre pairings lead to potentiation.
It has also been observed experimentally in the striatum, see~\citet{fino_bidirectional_2005} for example.

It corresponds to the case where  $\Phi_{p,1}{=}0$ and $\Phi_{d,2}{=}0$, and symmetric LTD rules  to $\Phi_{p,1}{=}\Phi_{p,2}{=}0$ and, finally,  symmetric LTP by $\Phi_{d,1}{=}\Phi_{d,2}{=}0$.   This is the motivation of the general setting defined in Equation~\eqref{eq:FPa}.

\item Pre/post-synaptic-only plasticity rules can also be expressed into this formalism.  These models include a  component to express the direct influence of the pre- or post-synaptic spikes on the plasticity without any interaction between the two spike trains. In that case, the kernel $\Gamma_a^{\textup{PA}_1}$ would have the following expression, for $m_1$, $m_2{\in}{\cal M}_p(\R_+)$ and $a{\in}\{p,d\}$,
\begin{multline}    \label{eq:FPa_prepost}
\Gamma_{a}^{\textup{PA}_1}(m_1,m_2)(\diff t) \steq{def}  \left(\int_{(0,t)}\Phi_{a,2}(t{-}s)m_2(\diff s)\right)\,m_1(\diff t)\\{+}
  \left(\int_{(0,t)}\Phi_{a,1}(t{-}s)m_1(\diff s)\right)\,m_2(\diff t) + D_{a,1}\,m_1(\diff t)+D_{a,2}\,m_2(\diff t),
\end{multline}
where the constants $D_{a,i}$, $a{\in}\{p,d\}$, $i{\in}\{1,2\}$, are non-negative.
\end{enumerate}

\begin{figure}[ht!]
\begin{subfigure}{\textwidth}
\fontsize{8pt}{8pt}\selectfont
\makebox[\textwidth][c]{
\def\svgwidth{0.55\textwidth}
\begingroup%
  \makeatletter%
  \providecommand\color[2][]{%
    \errmessage{(Inkscape) Color is used for the text in Inkscape, but the package 'color.sty' is not loaded}%
    \renewcommand\color[2][]{}%
  }%
  \providecommand\transparent[1]{%
    \errmessage{(Inkscape) Transparency is used (non-zero) for the text in Inkscape, but the package 'transparent.sty' is not loaded}%
    \renewcommand\transparent[1]{}%
  }%
  \providecommand\rotatebox[2]{#2}%
  \newcommand*\fsize{\dimexpr\f@size pt\relax}%
  \newcommand*\lineheight[1]{\fontsize{\fsize}{#1\fsize}\selectfont}%
  \ifx\svgwidth\undefined%
    \setlength{\unitlength}{184.75795316bp}%
    \ifx\svgscale\undefined%
      \relax%
    \else%
      \setlength{\unitlength}{\unitlength * \real{\svgscale}}%
    \fi%
  \else%
    \setlength{\unitlength}{\svgwidth}%
  \fi%
  \global\let\svgwidth\undefined%
  \global\let\svgscale\undefined%
  \makeatother%
  \begin{picture}(1,0.28923591)%
    \lineheight{1}%
    \setlength\tabcolsep{0pt}%
    \put(0,0){\includegraphics[width=\unitlength]{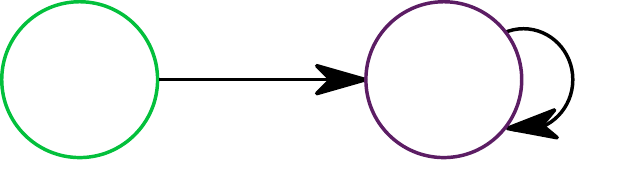}}%
    \put(0.12282163,0.1635348){\makebox(0,0)[t]{\lineheight{1.25}\smash{\begin{tabular}[t]{c}$\mathcal{N}_{\lambda}$, $t_{\mathrm{pre}}$\end{tabular}}}}%
    \put(0.69187463,0.1635348){\makebox(0,0)[t]{\lineheight{1.25}\smash{\begin{tabular}[t]{c}$\mathcal{N}_{\beta, X}$, $t_{\mathrm{post}}$\end{tabular}}}}%
    \put(0.12282163,0.0011602){\makebox(0,0)[t]{\lineheight{1.25}\smash{\begin{tabular}[t]{c}Pre-synaptic\end{tabular}}}}%
    \put(0.69187463,0.0011602){\makebox(0,0)[t]{\lineheight{1.25}\smash{\begin{tabular}[t]{c}Post-synaptic\end{tabular}}}}%
    \put(0.38664517,0.10302509){\color[rgb]{0,0,0}\makebox(0,0)[t]{\lineheight{1.25}\smash{\begin{tabular}[t]{c}$W(t)$\end{tabular}}}}%
    \put(0.98324983,0.1588526){\color[rgb]{0,0,0}\makebox(0,0)[t]{\lineheight{1.25}\smash{\begin{tabular}[t]{c}$-g(x)$\end{tabular}}}}%
  \end{picture}%
\endgroup%

}
\caption{A Simple Stochastic Model for a Synaptic Connection}
\label{figsub:model}
\end{subfigure}
\begin{subfigure}{\textwidth}
\input{Figures/Figure_PI.pgf}
\caption{Synaptic Plasticity Kernels for Pair-Based Rules\\
    (Top left) {\em STDP curve:} update of the potentiation (in red) and depression (in blue) kernels as a function of $\Delta t{=}t_{\textup{post}}{-}t_{\textup{pre}}$.
    Exponential STDP curves of the form $\Phi_{a,i}(\Delta t){=}B_{a,i}\exp({-}\gamma_{a,i} \Delta t)$ are used in this example.\\
    (Bottom left) {\em All-to-all pair-based rules:} all pairings of pre-synaptic (in green) and post-synaptic (in purple) spikes are taken into account for the synaptic updates.
    Grey arrows indicate the interactions between  the different spikes, see the associated updates as a function of the STDP curve above (blue and red points in (Top left)).\\
    (Top right) {\em Nearest neighbor symmetric pair-based rules:} for each pre-synaptic spike (in green), only the interaction with the previous post-synaptic spike (in purple) is taken into account for the synaptic update, and conversely for post-synaptic spikes.
    Grey arrows indicate the interactions between the different spikes.\\
    (Bottom right) {\em Nearest neighbor reduced symmetric pair-based rules:} only consecutive pairings of pre-synaptic (in green) and post-synaptic spikes (in purple) are taken into account for the synaptic update, and conversely for post-synaptic spikes.
    Grey arrows indicate the interactions between the different spikes.
}
\label{figsub:pair}
\end{subfigure}
\caption{Stochastic Models of STDP}
\label{fig:modelpair}
\end{figure}

\subsection*{Nearest Neighbor Symmetric Model}
In the {\em nearest neighbor symmetric} model, whenever one neuron spikes, the synaptic weight is updated by only taking into account the last spike of the other neuron, as can be seen in Figure~\ref{figsub:pair}~(top right). If the pre-synaptic neuron fires at time $t_{\rm pre}$, the contribution to the plasticity kernel is $\Phi_{a,2}(t_{\rm pre}{-}t_{\rm post})$ , where $t_{\rm post}$ is the last post-synaptic spike before $t_{\rm pre}$.

The corresponding kernels $\Gamma^{\textup{PS}}$ are defined by, for  $m_1$, $m_2{\in}{\cal M}_p(\R_+)$ and $a{\in}\{p,d\}$,
\begin{equation}    \label{eq:FPa_nearestPS}
  \Gamma_{a}^{\textup{PS}}(m_1,m_2)(\diff t) \steq{def}
  \Phi_{a,2}(t_0(m_2,t))\,m_1(\diff t){+}\Phi_{a,1}( t_0(m_1,t))\,m_2(\diff t),
\end{equation}
with the following definition, for $m{\in}{\cal M}_p(\R_+)$ and $t{>}0$,
\begin{equation}
\label{eq:t0}
t_0(m,t)=t-\sup\{s: s{<}t, m(\{s\}){\ne}0\},
\end{equation}
with the convention that $t_0(m,0){=}+\infty$. The quantity $t_0(m,t)$ is the delay between $t$ and the last point of $m$ before $t$.

\subsection*{Nearest Neighbor Reduced Symmetric Model}
For the {\em nearest neighbor reduced symmetric}, a pre-synaptic spike at $t$ is paired  with the last post-synaptic spike at $s{\le}t$, only if there are no pre-synaptic spikes in the time interval $(s,t)$, and similarly for  post-synaptic spikes.  See Figure~\ref{figsub:pair}~(bottom right).

Accordingly, the kernels $\Gamma^{\textup{PR}}$ are defined, for  $m_1$, $m_2{\in}{\cal M}_p(\R_+)$ and $a{\in}\{p,d\}$, by
\begin{multline}    \label{eq:FPa_nearestPR}
  \Gamma_{a}^{\textup{PR}}(m_1,m_2)(\diff t) \steq{def}
\left(\Phi_{a,2}(t_0(m_2,t))\ind{t_0(m_2,t){\le}t_0(m_1,t)}\right)\,m_1(\diff t)\\{+}
\left(\Phi_{a,1}( t_0(m_1,t))\ind{t_0(m_1,t){\le}t_0(m_2,t)}\right)\,m_2(\diff t),
\end{multline}
with same notations as in~\eqref{eq:FPa_nearestPS}. For $t{>}0$, the inequality $t_0(m_2,t){<}t_0(m_1,t)$ is equivalent to the relation $m_1((t_0(m_2,t),t)){=}0$ so that there is a unique point of $m_1$ paired to $t_0(m_2,t)$ as expected, and similarly by switching $m_1$ and $m_2$. The updates of Relation~\eqref{eq:FPa_nearestPR} are therefore done only for consecutive pre- and post-synaptic spikes.

\subsubsection{\sc Calcium-Based  Models} \label{secsecsec:cbm}
Pair-based models can be characterized as phenomenological models of STDP in the sense that experimental STDP curves are taken as a core parameter of the models.
Another important class of synaptic models are derived from biological phenomenons and aims at reproducing experimental STDP curves using simple biological models.
 A common hypothesis is to use the calcium concentration in the post-synaptic neuron as a key parameter to model STDP, see~\citet{shouval_unified_2002} and~\citet{graupner_calcium-based_2012}.
Several biophysical models have studied the link between calcium concentration, and its direct implication on the dynamics of plasticity.
A  calcium-based model with saturation mechanisms has investigated the dependency on the number of pairings and the existence of different mechanisms for plasticity in~\citet{vignoud_interplay_2018}.

For these models,  synaptic plasticity is expressed as a functional of the post-synaptic calcium concentration.
For $m_1$, $m_2{\in}{\cal M}_p(\R_+)$,
the points of $m_1$, resp. $m_2$, elicit calcium transfers of amplitudes $C_1$, resp. $C_2$, followed by an exponential decay with rate~$\gamma$. If $(C_m(t))$  is the process of the calcium concentration associated to the couple $m{=}(m_1,m_2)$, it is therefore the  solution of the differential equation
\[
\diff C_m(t)= {-}\gamma C_m(t)\diff t+C_1 m_1(\diff t)+C_2 m_2(\diff t),
\]
with  some fixed initial condition.
By Lemma~\ref{FiltLem}, it can be expressed as
\begin{equation}     \label{eq:CaC}
C_m(t)\steq{def}C_m(0)e^{-\gamma t}{+}C_1\int_{(0,t]} e^{{-}\gamma(t{-}s)}m_1(\diff s) {+}C_2\int_{(0,t]} e^{{-}\gamma(t{-}s)}m_2(\diff s).
\end{equation}
The mechanisms for potentiation, resp. depression, are triggered depending on the calcium concentration. For $a{\in}\{p,d\}$, the plasticity kernel $\Gamma_a^{\textup{C}}$ is defined by,
\begin{equation}     \label{eq:FCa}
  \Gamma_{a}^{\textup{C}}\left(m_1,m_2)(\diff t\right) \steq{def} h_a(C_m(t)) \,\diff t,
\end{equation}
for some non-negative function $h_a$ on $\R_+$. The function $h_a$ is usually a threshold function of the type
\begin{equation}     \label{eq:CaThresh}
h_{a}(x)\steq{def} B_a\ind{x{\ge}\theta_a},\quad x{\ge}0,
\end{equation}
for some $B_a{\in}\R_+$ and $\theta_a{\geq}0$, as done in~\citet{graupner_calcium-based_2012}. In that case, the process $(\Omega_a(t))$ associated to $\Gamma_{a}^C$ has therefore an impact on the synaptic weight as soon as the concentration of calcium is above level $\theta_a$.

Further examples of STDP rules are presented in Section~\ref{AsecMod}.

\subsection{The Plasticity Process}
\label{secsec:plasticityprocess}
This section is devoted to the formal definition of the stochastic process describing the time evolution of the synaptic weight.
\begin{definition}
\label{def:proc}
The stochastic process $(X(t),\Omega_p(t) ,\Omega_d(t), W(t))$  with the initial state $(x_0,\omega_{0,p}, \omega_{0,d}, w_0)$, is the solution in ${\cal D}(\R_+,\R{\times}\R_+^2{\times}K_W)$ of the SDEs,  for $t{>}0$,
\begin{equation}\label{Syst}
  \begin{cases}
\quad  \diff X(t) = \displaystyle -X(t)\diff t+W(t{-})\mathcal{N}_\lambda(\diff t)-g(X(t-))\mathcal{N}_{\beta,X}(\diff t),\\
\quad \diff \Omega_a(t) = \displaystyle -\alpha\Omega_a(t)\diff t + \Gamma_a(\mathcal{N}_\lambda,{\cal N}_{\beta,X})(\diff t),\quad a{\in}\{p,d\},\\
\quad    \diff W(t) = \displaystyle M\left(\Omega_p(t),\Omega_d(t),W(t)\right)\diff t,
  \end{cases}
\end{equation}
where, $\Gamma_p$ and $\Gamma_d$ are plasticity kernels and ${\cal N}_{\beta,X}$ is the point process defined by Relation~\eqref{Nbeta} and the function $M$ is expressed by Relation~\eqref{eq:M0}.
\end{definition}
The system~\eqref{Syst} can be interpreted as fixed point equation for the process $(X(t))$ with an intricate dependence due to the point process ${\cal N}_{\beta,X}$ as an argument of the plasticity kernels.
Theorem~\ref{th:ExistTheo} gives an existence and uniqueness result for the solutions of Equations~\eqref{Syst}.   We now introduce the main assumptions on the parameters of our model which will be used throughout this paper.

Examples of different dynamics are presented in
Section~\ref{ap:figures}, for pair-based in Figure~\ref{figS:pairbased} and calcium-based in Figure~\ref{figS:calciumbased}.

\medskip
\noindent
{\bf Assumptions~A}
\begin{enumerate}
\item  {\sc Firing Rate Function}.\label{AsF}\\ $\beta$ is a non-negative, continuous function on $\R$ and $\beta(x){=}0$ for $x{\le}{-}c_\beta{\le}0$.
  \medskip
\item  {\sc Drop of Potential after Firing}.\label{AsD}\\ $g$ is continuous on $\R$ and $0{\le}g(x){\le} \max(c_g,x)$ holds for all $x{\in}\R$, for $c_g{\ge}0$.
  \medskip
\item {\sc Dynamic of Plasticity.}\label{AsM}\\ The function $M$ is
  such that, for any $w{\in}K_W$ and any \cadlag piecewise-continuous functions $h_1$ and $h_2$ on $\R_+$, the ODE
\begin{equation}\label{ODEM}
  \frac{\diff w(t)}{\diff t}{=}M(h_1(t),h_2(t),w(t)) \text{ with } w(0){=}w,
\end{equation}
for all points of continuity of $h_1$ and $h_2$, has a unique continuous solution $(S[h_1,h_2](w,t))$ in $K_W$.
\end{enumerate}
\begin{theorem}\label{th:ExistTheo}
Under Assumptions~A, the system~\eqref{Syst} has a unique  \cadlag adapted solution with initial state $(x_0,\omega_{0,p}, \omega_{d,0},w_0)$ in $\R{\times}\R_+^2{\times}K_W$.
\end{theorem}
\begin{proof}
The construction is done on the successive intervals between two consecutive instants of jump of the system. The non-decreasing sequence $(s_n)$ of these instants is defined by induction.
  
The first jump of $(X(t))$ occurs at time $s_1$ and is defined as the minimum of the first jumps of the processes
\begin{equation}\label{TL1}
  ({\cal N}_\lambda((0,t])) \text{ and  } \left(\int_{(0,t]}{\cal P}\left(\rule{0mm}{4mm}\left(\rule{0mm}{3mm}0,\beta\left(x_0e^{-u}\right)\right],\diff u\right)\right).
\end{equation}
With Relation~\eqref{ODEM}, for $0{\le}t{<}s_1$, we set $X(t){=}x_0e^{-t}$ and  $W(t){=}S[\Omega_p^1,\Omega_d^1](w_0,t)$, with
\[
\Omega_a^1(t) \steq{def} \omega_{0,a}{+}\int_{(0,t)} e^{-\alpha(t-s)} \Gamma_a(\underline{0},\underline{0})(\diff s),\quad a{\in}\{p,d\},
\]
and $W(s_1){=}W(s_1{-})$, where $\underline{0}$ is the null point process.
\begin{enumerate}
\item If $s_1$ is the first point of  ${\cal N}_\lambda$, define
  \[
f_1\steq{def}{+} \text{ and }  X(s_1)=x_0e^{-s_1}{+}W(s_1{-}).
  \]
\item If $s_1$ is the first point of the second point process of Relation~\eqref{TL1}, set
  \[
f_1\steq{def}{-}  \text{ and }   X(s_1)=x_0e^{-s_1}{-}g\left(x_0e^{-s_1}\right).
  \]
\end{enumerate}
The mark $f_1$ indicates the nature of the jump occurring at time $s_1$, i.e. if the spike was fired by the pre- or post-synaptic neuron.

The process $(X(t),\Omega_p^1(t),\Omega_d^1(t),W(t))$ satisfies the equations~\eqref{Syst} on the time interval $[0,s_1]$ and, by Relation~\eqref{eq:ft}, $s_1$ is a stopping time with respect to~$({\cal F}_t)$.

Assume by induction that, for $n{\ge}0$, the variables $(s_k,f_k, 1{\le}k{\le}n)$ and the adapted \cadlag process $(X(t),W(t), t{\in}[0,s_n])$ are defined, and $s_n$ is a stopping time. For $a{\in}\{p,d\}$, let
\begin{equation}\label{OmP}
\Omega_a^{n+1}(t) \steq{def} \omega_a{+}\int_{(0,t)} e^{-\alpha(t-s)} \Gamma_a\left(\sum_{k=1}^n\delta_{s_k}\ind{f_k{=}+},\sum_{k=1}^n\delta_{s_k}\ind{f_k{=}-}\right)(\diff s).
\end{equation}
In Definition~\ref{def:gamma}, the ${\cal G}_t{\otimes}{\cal G}_t$
measurability property, gives that for any $n{\ge}1$ and $k{<}n$, the
process $(\Omega_a^{j}(t))$ does not depend on the index $j{\in}\{k,\ldots,n\}$ on $[0,s_k]$.
The instant $s_{n+1}{>}s_n$ is defined as the minimum of the first jumps of the two point processes,
\begin{equation}\label{TL2}
  ({\cal N}_\lambda([s_n,t]), t{>}s_n), \left(\int_{[s_n,t]}{\cal P}\left[\left(0,\beta\left(X(s_n)e^{-(u-s_n)}\right)\right],\diff u\right], t{>}s_n\right).
\end{equation}
The fact that $s_n$ is a stopping time and the strong Markov property of the Poisson processes ${\cal N}_\lambda$ and ${\cal P}$ give that $s_{n+1}$ is also a stopping time.
For $s_n{\le}t{<}s_{n+1}$, set
\[
W(t){=}S[\Omega_p^{n+1},\Omega_d^{n+1}](W(s_n),t{-}s_n) \text{ and }X(t)\steq{def}X(s_n)e^{-(t-s_n)},
\]
and $W(s_{n+1}){=}W(s_{n+1}{-})$, and 
\begin{enumerate}
\item if $s_{n+1}$ is a point of  ${\cal N}_\lambda$, define $f_{n+1}{=}{+}$, and 
  \[
  X(s_{n+1}){\steq{def}}X(s_n)e^{-(s_{n+1}-s_n)}{+}W(s_{n+1}{-}),
    \]
\item Otherwise, we set $f_{n+1}{=}{-}$, and 
  \[
  X(s_{n+1}){\steq{def}}X(s_n)e^{-(s_{n+1}-s_n)}{-}g\left(X(s_n)e^{-(s_{n+1}-s_n)}\right).
  \]
\end{enumerate}
We have thus defined by induction  a stochastic process $(X(t),W(t))$ on sequence of  time intervals $(s_n,s_{n+1})$, $n{\ge}1$.
We now prove that the process is defined on the whole real half-line, i.e. that the sequence $(s_n)$ is  almost surely converging to infinity.
This is the object of the following lemma.

\medskip
\noindent
\begin{lemma}[Non-Explosive Behavior]
\label{lemma:nonexplo}
Under Assumptions~A, the sequence of successive jump instants $(s_n)$ is almost surely converging to infinity. 
\end{lemma}
\begin{proof}
Denote by  ${\cal E}_0$ the  event  where the sequence $(s_n)$ is bounded and assume that it has a positive probability. 
On the event ${\cal E}_0$, almost surely, only a finite number of points of the Poisson process ${\cal N}_\lambda$  may be points of the sequence $(s_n)$.  Therefore, there  exists some $N_0{\in}\N$ and  a subset ${\cal E}_1$ of ${\cal E}_0$ of  positive probability such that for $n{\ge}N_0$, one has $f_n{=}-$, i.e. the jumps are due to the second point process of Relation~\eqref{TL2} after time~$s_{N_0}$.

On the event ${\cal E}_1$, for $n{\ge}N_0$ one has
$X(s_n{-}){<}|X(s_{N_0})|$,
almost surely, because $(|X(t)|)$ can only decrease when there are no pre-synaptic spikes. Consequently, as $\beta(x)$ is null for $x{<}{-}c_{\beta}$, we have that $\max(\beta(X(t)){:}t{>}s_{N_0}){<}+\infty$. Therefore, the successive jump instants $(s_n,n{\ge}N_0)$ cannot stay bounded on the event  ${\cal E}_1$. This is a contradiction. The sequence $(s_n)$ is therefore converging to infinity almost surely.
\end{proof}

A direct consequence of this result is that, from the very definition of the sequence $(s_n)$,  for any $t{>}0$, there exists $n_0$ such that if $n{\ge}n_0$ then
\[
\sum_{k=1}^n\delta_{s_k}\ind{f_k{=}+}\cap[0,t]={\cal N}_\lambda \cap[0,t] \text{ and }\sum_{k=1}^n\delta_{s_k}\ind{f_k{=}{-}}\cap[0,t]={\cal N}_{\beta,X}\cap[0,t],\quad\text{a.s.},
\]
recall that ${\mu}\cap[0,t]$ is the measure $\mu{\in}{\cal M}(\R_+)$ restricted to the interval $[0,t]$.
For $a{\in}\{p,d\}$, again with  the ${\cal G}_t{\otimes}{\cal G}_t$-measurability property of  plasticity kernels, the quantity
\[
\Omega_a^{n}(t) {=} \Omega_a(0){+}\int_{(0,t)} e^{-\alpha(t-s)} \Gamma_a\left({\cal N}_\lambda,{\cal N}_{\beta,X}\right)(\diff s)
\]
is constant for  $n{\ge}n_0$, it is defined as $\Omega_a(t)$.
Furthermore, for $s{\le}t$ and $n{\ge}n_0$,
\[
\diff W(s) =  M\left(\Omega_p^n(t),\Omega_d^n(s), W(s)\right)\diff s=M\left(\Omega_p(t),\Omega_d(s), W(s)\right)\diff s.
\]
We have thus the existence of a solution to Relation~\eqref{Syst}. The uniqueness is clear on any time interval $[0,s_n]$, $n{\ge}1$, and therefore almost surely  on $\R_+$.

\end{proof}

\section{Markovian Plasticity Kernels}
\label{sec:markov}\label{secsec:kenerlM}
In this section we introduce an important subclass~$(\mathcal{M})$ of plasticity kernels that leads to a Markovian formulation of the whole plasticity process.
In this context, it turns out  that the associated synaptic weight process $(W(t))$ can be investigated with a scaling approach which is often used, sometimes implicitly, in the literature of physics in neuroscience. As it will be seen,  plasticity kernels of pair-based models of Section~\ref{secsecsec:pairbased} and  of calcium-based models of Section~\ref{secsecsec:cbm} are of class ${\cal M}$. The follow-up paper \citet{robert_stochastic_2020_2} is devoted to the scaling analysis of these plasticity kernels.
\begin{definition}[Kernels of Class~$(\mathcal{M})$]
\label{def:classM}
A plasticity kernel $\Gamma$ is of class~$(\mathcal{M})$  if, for $m_1$, $m_2{\in}{\cal M}_p(\R_+)$,
\begin{equation}\label{KerM}
\Gamma(m_1,m_2)(\diff t)=n_{0}(z(t))\diff t+n_{1}(z(t{-}))m_1(\diff t) +n_{2}(z(t{-}))m_2(\diff t),
\end{equation}
where
\begin{enumerate}
\item For $i{=}0$, $1$, $2$, $n_{a,i}$  is a non-negative measurable function on $\R_+^\ell$,  where $\ell{\in}\N_{\ast}$;
\item $(z(t))$ is a \cadlag function with values in $\R_+^\ell$, solution of the SDE
\begin{equation}\label{KerZ}
    \diff z(t) =({-}\gamma\odot z(t){+}k_0)\diff t +k_1(z(t{-}))m_1(\diff t)+k_2(z(t{-}))m_2(\diff t),
\end{equation}
\begin{itemize}
\item $\gamma{\in}\R_{+}^\ell$, $a{\odot}b{=}(a_i{\times}b_i)$  if $a{=}(a_i)$ and $b{=}(b_i)$ in $\R_+^\ell$;
\item $k_0{\in}\R_+^\ell$ is a constant and  $k_{1}$ and   $k_{2}$  are measurable functions from $\R_+^\ell$ to $\R^\ell$. Furthermore, the $(k_i)$  are such that the function $(z(t))$ has values in $\R_+^{\ell}$ whenever $z(0){\in}\R_+^{\ell}$.
\end{itemize}
\end{enumerate}
\end{definition}
It is important to note that the function $(z(t))$ is a functional of the pair $(m_1,m_2)$.
The fact that $z(t)$ stay non-negative is an important feature of class~$(\mathcal{M})$ kernels. For example,  we  may have functions $k_1$ or $k_2$ of the form,
\[
 k_i(z) = B_i{-} b_{i}{\odot}z
\]
where $B_i{\in}\R_+^{\ell}$, and $b_{i}{\in}\{0,1\}^{\ell}$.

If $\Gamma$ is of class~$(\mathcal{M})$ and $(z(t))$ is its associated \cadlag process, with
Relation~\eqref{KerZ} it is easily seen  that, for any $t{>}0$, the functional
\[
\begin{cases}
({\cal M}_p(\R_+)^2,{\cal G}_t{\otimes}{\cal G}_t)&{\longrightarrow } \ ({\cal M}_+([0,t]),{\cal B}({\cal M}_+([0,t])))\\
\phantom{{\cal M}_p(\R_+)^2}(m_1,m_2)&{\longrightarrow} \Gamma (m_1,m_2)(\diff u \cap [0,t])
\end{cases}
\]
is indeed ${\cal G}_t$-measurable, where $({\cal G}_t)$ is the filtration of Definition~\eqref{def:gamma}.
\begin{proposition}[A Markovian Formulation of Plasticity]\label{MPlastprop}
If $\Gamma_a$, $a{\in}\{p,d\}$, are plasticity kernels of class~$(\mathcal{M})$ associated  to $(n_{a,i},k_i), i{\in}\{0,1,2\}$, $a{\in}\{p,d\}$ and $\gamma{\in}\R_+^\ell$ and under Assumptions~A, the solution of Relations~\eqref{Syst} of Theorem~\ref{th:ExistTheo} is such that the stochastic process
$(U(t))\steq{def}(X(t),Z(t),\Omega_{p}(t),\Omega_{d}(t),W(t))$
is a Markov process on $\mathcal{S}_{\mathcal{M}}(\ell)\steq{def}\R{\times}\R_+^{\ell}{\times}\R_+^2{\times}K_W$, solution of the SDE,
\begin{equation}\label{eq:markov}
\begin{cases}
    \quad\diff X(t) &\displaystyle = -X(t)\diff t+W(t)\mathcal{N}_{\lambda}(\diff t)-g\left(X(t-)\right)\mathcal{N}_{\beta,X}\left(\diff t\right),\\
    \quad\diff Z(t) &\displaystyle {=}   (-\gamma\odot Z(t)+ k_0)\diff t\\ &\hspace{2cm}+k_1(Z(t{-}))\mathcal{N}_{\lambda}(\diff t)+k_2(Z(t{-}))\mathcal{N}_{\beta,X}(\diff t),\\
    \quad \diff \Omega_a(t)&\displaystyle ={-}\alpha\Omega_a(t)\diff t {+}n_{a,0}(Z(t))\diff t\\&\hspace{0cm}+
    n_{a,1}(Z(t{-}))\mathcal{N}_{\lambda}(\diff t){+}n_{a,2}(Z(t{-}))\mathcal{N}_{\beta,X}(\diff t),\quad a{\in}\{p,d\},\\
    \quad\diff W(t) &\displaystyle = M\left(\Omega_{p}(t),\Omega_{d}(t), W(t) \right)\diff t.
\end{cases}
\end{equation}
\end{proposition}
\begin{proof}
Theorem~\ref{th:ExistTheo} shows the existence and uniqueness of such a process $(U(t))$. The process $(U(t))$ is a piecewise deterministic Markov process in the sense of~\citet{davis_markov_1993} and consequently has the Markov property. See Chapter~2 of~\citet{davis_markov_1993}. An expression of its infinitesimal generator is given in Section~\ref{ap:generator} of \SMS.
\end{proof}
It should be noted that, due to the dimension of the state space,  the Markov property of $(U(t))$  cannot be really used in practice in our analysis. The representation in terms of SDEs in Relation~\eqref{eq:markov} turns out to be useful in the scaling approach presented in~\citet{robert_stochastic_2020_2}.

\subsection*{Motivation for Markovian Kernels}
The processes $(\Omega_p(t))$ and $(\Omega_d(t))$ determining the synaptic plasticity depend on the process $(Z(t))$ in a non-linear way.
The coordinates of $(Z(t)){=}(Z_i(t))$ may be interpreted as the concentration of chemical components created/suppressed by pre-synaptic and/or post-synaptic spikes,  with some leaking mechanism.
Calcium is such an example, see Relation~\eqref{eq:CaC}.
A simple case is when each coordinate of $(Z(t))$ is associated either to pre- or post-synaptic spikes,  i.e. it satisfies
\[
\diff Z_i(t)= {-}\gamma_i Z_i(t)\diff t +B_i {\cal N}_{\lambda}(\diff t) \text{ or } \diff Z_i(t)= {-}\gamma_i Z_i(t)\diff t +B_i {\cal N}_{\beta, X}(\diff t).
\]
Moreover, if $Z_i$ needs to be reset to $B_i$ when one of the neurons spikes, we just need to replace $B_i$ by $B_i{-}Z_i(t{-})$  in these equations. 

We now show that calcium-based models and several pair-based models, can be represented in such a setting, i.e.\ that their plasticity kernels are of class~$(\mathcal{M})$.
\subsection{Examples}
\subsubsection{Calcium-Based Models}
\label{secsecsec:cbmclassM}
For this set of models, the class $(\cal{M})$ property is fairly clear. Relations~\eqref{eq:CaC} and~\eqref{eq:FCa} give that,  for $a{\in}\{p,d\}$ and $m_1$, $m_2{\in}{\cal M}_p(\R_+)$,
\[
\Gamma_{a}^C\left(m_1,m_2)(\diff t\right) \steq{def} h_a(C_m(t))\diff t,
\]
where, if $m{=}(m_1,m_2)$, $(C_m(t))$ is a \cadlag solution of the differential equation
\[
\diff C_m(t)= {-}\gamma C_m(t)\diff t+C_1 m_1(\diff t)+C_2 m_2(\diff t).
\]
The process $(Z(t))$ is simply the one-dimensional process $(C_{{\cal N}_{\lambda}, {\cal N}_{\beta,X}}(t))$.
Markovian dynamics of the calcium-based model are illustrated in Figure~\ref{figS:calciumbased}-(a).

\subsubsection{Pair-Based Models}
\label{secsecsec:pbsclassM}
Several kernels associated to pair-based models defined by Relation~\eqref{eq:FPa} are also of class~$(\mathcal{M})$. This type of Markov property has been mentioned in~\citet{morrison_phenomenological_2008}.
Markovian  models including STDP models described in Section~\ref{secsec:plasticitykernel} are presented in Section~\ref{ap:generator} of \SMS.

\subsubsection{All-to-all Model}\label{STDPPA}
The class~$({\cal M})$ holds when the STDP functions $\Phi$  are exponential, i.e.\   when,  for $a{\in}\{p,d\}$ and $i{\in}\{1,2\}$,
\[
\Phi_{a,i}(t){=}B_{a,i}\exp({-}\gamma_{a,i} t), \quad t{\ge}0.
\]
with $B_{a,i}{\in}\R_+$ and $\gamma_{a,i}{>}0$.
For $m_1$ and $m_2{\in}{\cal M}_p(\R_+)$, denote by $(z_{a,i}(t))$,  the \cadlag solution of the differential equation
\[
\diff z_{a,i}(t)={-}\gamma_{a,i}z_{a,i}(t)\diff t +B_{a,i}m_i(\diff t),
\]
with $z_{a,i}(0){=}0$.  Lemma~\ref{FiltLem} gives the relation
\[
z_{a,i}(t)=B_{a,i}\int_{(0,t]} e^{-\gamma_{a,i}(t-s)}m_i(\diff s).
  \]
The  process $(z(t))$ is then defined as $(z_{p,1}(t),z_{p,2}(t),z_{d,1}(t),z_{d,2}(t))$. 
The plasticity kernel of this model, see Relation~\eqref{eq:FPa}, can be expressed as 
\[
\Gamma_{a}^{\textup{PA}}(m_1,m_2) = n_{a,1}(z(t-))m_1(\diff t){+}n_{a,2}(z(t-))m_2(\diff t),
 \]
the functions $(n_{a,i})$ are defined by, for $z{=}(z_{a,i}){\in}\R_+^4$,  $n_{a,1}(z){=}z_{a,2}$ and $n_{a,2}(z){=}z_{a,1}$. 

An example of dynamics with plasticity kernels and associated Markov process $(Z_{a,i})$ is presented in Figure~\ref{figS:pairbased}-(a).
Similar models, using auxiliary processes $(Z_{a,i})$ can be devised for nearest STDP rules. See Section~\ref{ap:generator} of \SMS, Figure~\ref{figS:pairbased}-(b) for the nearest neighbor symmetric STDP and Figure~\ref{figS:pairbased}-(c) for the nearest neighbor reduced symmetric STDP.


\subsubsection{Nearest Neighbor Models}\label{STDPNS}\label{STDPNR}
For $m{\in}{\cal M}_p(\R_+)$ and $t{>}0$,  the variable $t_0(m,t)$ of Relation~\eqref{eq:t0} used in the two models presented in Section~\ref{secsecsec:pairbased},  
\[
t_0(m,t)=t-\sup\{s: s{<}t, m(\{s\}){\ne}0\},
\]
can be expressed as the solution $(z_m(t))$ of the differential equation,
\[
\diff z_m(t)=\diff t {-}z_m(t-)m(\diff t),
\]
with $z_m(0){=}0$.

For $m_1$ and $m_2{\in}{\cal M}_p(\R_+)$, we define $(z(t)){=}(z_{m_1}(t),z_{m_2}(t))$, Relation~\eqref{KerZ} holds with
$\gamma{=}(0,0)$, $k_0{=}(1,1)$ and, for $z{=}(z_1,z_2)$,  $k_1(z){=}(-z_1,0)$ and $k_2(z){=}(0,-z_2)$.

In this setting, both nearest models are of  class ${\cal M}$:
\begin{itemize}
\item The nearest neighbor symmetric model, $
  \Gamma_{a}^{\textup{PS}}$ of
  Relation~\eqref{eq:FPa_nearestPS},\\ with $n_{a,0}(z){=}0$,
  $n_{a,1}(z){=} \Phi_{a,2}(z_2)$ and $n_{a,2}(z){=} \Phi_{a,1}(z_1)$.
 \item The nearest neighbor reduced symmetric model, $ \Gamma_{a}^{\textup{PR}}$ of Relation~\eqref{eq:FPa_nearestPR},\\ with
$  n_{a,0}(z){=}0$, $n_{a,1}(z){=} \Phi_{a,2}(z_2)\ind{z_{2}{\le}z_1}$
   and $n_{a,2}(z){=} \Phi_{a,1}(z_1)\ind{z_{1}{\le}z_2}$.
\end{itemize}

\subsection{Extensions}
For all-to-all models, the exponential STDP function  allows the representation of time evolution of plastic synapticity with a finite-dimensional process $(Z(t))$ and, therefore, the associated kernels are of class ${\cal M}$.

For a general function $\Phi$, it is however possible to express the system as a Markovian system, by taking the instants of all past instants of spikes
\[
(Z_{1,k}(t))\steq{def} \left(t_k({\cal N}_\lambda,t), k{\ge}0\right) \text{ and } (Z_{2,k}(t))\steq{def}\left(t_k({\cal N}_{\beta,X},t), k{\ge}0\right)
\]
with, for $k{\ge}0$, $m{\in}{\cal M}_p(\R_+)$ and $t{>}0$,
\[
t_k(m,t)=t{-}\sup\{s{\le}t: m([s,t]){>}k\}
\]
$t_k(m,t)$ represent the time between $t$ and the $k$th last spike of $m$. In an analogous way as for Definition~\eqref{eq:t0}, we have, for $k{\ge}1$, 
\[
\diff t_k(m,t) = \diff t {+} \left(t_k(m,t{-})-t_{k-1}(m,t{-})\right)m(\diff t),
\]
so that the processes $(Z_{i,k}(t),k{\ge}0)$, $i{\in}\{1,2\}$ would satisfy SDE as in Relation~\eqref{eq:markov}.
Keeping track of \emph{all} instant previous spikes, we can express the plasticity kernels with an infinite dimensional Markovian process. 
Unfortunately, there are fewer results, concerning  equilibrium distributions for example, in such a context. This is why we restrict our study to finite-dimensional systems.

\subsection*{Markov Processes Associated to Cellular Processes}\label{sec:Cell}
When the plasticity process $(W(t))$ is constant and equal to $w{\in}\R$, the associated solution $(X^w(t),Z^w(t))$ of the first two SDEs  of Relations~\eqref{eq:markov} is clearly a Markov process driven by the pre- and post-synaptic spikes. The invariant distribution of this  Markov process plays in important role in the scaling analysis of the process $(W(t))$ developed in~\citet{robert_stochastic_2020,robert_stochastic_2020_2}.  For reasons explained in the introduction of~\cite{robert_stochastic_2020_2}, these processes are referred to as {\em fast processes}.

It is easily seen that its infinitesimal generator is defined by,
if $f{\in}\mathcal{C}_b^1(\R{\times}\R_+^\ell)$ and $v{=}(x,z){\in}\R{\times}\R_+^{\ell}$, then
\begin{multline}\label{BFGen}
B^F_w(f)(v)\steq{def}{-}x\frac{\partial f}{\partial x}(x,z){+}\croc{{-}\gamma{\odot}z+k_0,\frac{\partial f}{\partial z}(x,z)}
\\+\lambda \left(\rule{0mm}{4mm}f(x{+}w,z{+}k_1(z)){-}f(v)\right)
+\beta(x)\left(\rule{0mm}{4mm}f(x{-}g(x),z{+}k_2(z)){-}f(v)\right),
\end{multline}
with
\[
\frac{\partial f}{\partial z}(x,z){=}\left(\frac{\partial f}{\partial z_i}(x,z), i{\in}\{1,\ldots,\ell\}\right).
\]
Examples of fast processes for classical STDP rules are presented in Section~\ref{ap:generator} of \SMS. 
The following proposition is proved in Section~5 of~\citet{robert_stochastic_2020}.
\begin{proposition}\label{InvPropFP}
  Under the Assumptions~A-a and~A-b and if the functions $k_1$ and $k_2$ are bounded and all coordinates of $\gamma$ are positive then the Markov process $(X^w(t),Z^w(t))$ has a unique invariant distribution $\Pi_w$.
\end{proposition}
The explicit expression of $\Pi_w$ is not known in general. For several STDP models, like calcium-based models, this is a limitation for a detailed analysis of the plasticity process $(W(t))$. See Section~4 of~\cite{robert_stochastic_2020_2}. The next section is devoted to a class of discrete models of synaptic plasticity  for which the corresponding $\Pi_w$ has an explicit expression for the analogue of  calcium-based models.

\section{Discrete Models of  STDP Rules}
\label{ap:stochqueu}

In this section, we introduce a discrete model of plasticity associated to Relation~\eqref{eq:markov}, where  the membrane potential $X$, the cellular processes  $Z$ and the synaptic weight $W$ are integer-valued variables. It amounts to represent these  quantities as multiple of a ``quantum'',  instead of a continuous variable.
For example, pre-/post-synaptic receptors (like the AMPA receptor for example) have a measurable influence on the membrane potential, where one quantum would represent the influence of a single receptor.
This is a biologically plausible assumption for potential and cellular processes.  The leaking mechanism (${-}a U(t){\diff}t$ in the continuous model, $U{\in}\{X,Z,W\}$ and $a{>}0$, in the SDEs) is represented by the fact that each quantum leaves the cell/synapse at rate $a$.

\begin{equation}\label{eq:discrete}
\begin{cases}
 \diff X(t) &= \displaystyle{-}{\cal N}_{I,X}(\diff t) +W(t{-})\mathcal{N}_{\lambda}(\diff t)-{\cal N}_{I,\beta X}(\diff t),\\
\diff Z_j(t) &\displaystyle {=}   {-}{\cal N}_{I,\gamma_j Z_j}(\diff t)+k_{0,j}(Z(t{-})){\cal N}_{1}^{j}(\diff t){+}k_{1,j}(Z(t{-}))\mathcal{N}_{\lambda}(\diff t)\\ &\displaystyle\hspace{1cm}+k_{2,j}(Z(t{-})){\cal N}_{I,\beta X}(\diff t),\quad j{=}1,\ldots,\ell,\\
 \diff W(t) &=\displaystyle -{\cal N}_{I,\mu W}(\diff t) {+}A_p{\cal N}_{I,\Omega_{p}}(\diff t){-}A_d\ind{W(t-){\ge}A_d}{\cal N}_{I,\Omega_{d}}(\diff t).
\end{cases}
\end{equation}
The processes $(\Omega_{a}(t))$, $a{\in}\{p,d\}$ satisfy the same SDE as in Relation~\eqref{eq:markov}, the functions $n_{a,i}$ and $k_i$, $i\in\{0,1,2\}$  are defined on $\N^\ell$ with values in $\N^\ell$. The variables $A_p$ and $A_d$ are integers and $\gamma{=}(\gamma_j){\in}\R_+^\ell$. 

For $\xi{>}0$, ${\cal N}_\xi$, resp. $({\cal N}_{\xi}^i)$ , is a Poisson process on $\R_+$ with rate $\xi$, resp. independent i.i.d. sequences of such point processes.
As before, with Relation~\eqref{Nbeta} and $I(x){=}x$ and a process $(U(t))$, the notation ${\cal N}_{I,U}(\diff t)$ stands for ${\cal P}\left((0,U(t-)),\diff t\right)$, where ${\cal P}$ is a Poisson process in $\R_+^2$ with rate $1$. We have in particular $$\P({\cal N}_{I,U}(\diff t){\ne}0{\mid}U(t{-})){=}U(t{-})\diff t{+}o(\diff t).$$ All Poisson processes are assumed to be independent.

We have taken $g(\cdot)$ as the constant function equal to $1$. As it can be seen, the firing rate in the evolution of $(X(t))$ is the linear function $x{\mapsto}\beta x$. The time evolution of the discrete random variable $(W(t))$ is driven by two inhomogeneous Poisson processes, one for potentiation and the other for depression.

As before we define $(X^w(t),Z^w(t))$ as the Markov process $(X(t),Z(t))$ when $(W(t))$ is constant and equal to $w{\in}\N$. 
If $Q{=}(q((x,z),(x',z')))$ is the jump matrix of $(X^w(t),Z^w(t))$, we have,
\[
\begin{cases}
  q((x,z),(x{-}1,z)){=} x,\;\;  q((x,z),(x,z{+}k_0(z))){=} 1\\
 q((x,z),(x{+}w,z{+}k_1(z))){=} \lambda,\;\;  q((x,z),(x{-}1,z{+}k_2(z))){=}\beta x,\\
 q((x,z),(x,z{-}e_i)){=}\gamma_iz_i, \quad i{\in}\{1,\ldots,\ell\}.
\end{cases}
\]
where $e_i$ is the $i$th unit vector of $\N^\ell$. If $f$ is a function on $\N{\times}\N^\ell$,  with the notation $\nabla_{(a,b)}(f)(v){=}(f(v{+}(a,b)){-}f(v))$, for $v$, $(a,b){\in}\Z^{\ell+1}$, $Q$ can be expressed as 
\begin{multline*}
Q(f)(x,z)\steq{def}\sum_{(x',z')} q((x',z'),(x,z))f(x',z')\\{=} x\nabla_{(-1,0)}(f)(x,z){+}\sum_{j=1}^\ell \gamma_j z_j\nabla_{(0,-e_j)}(f)(x,z){+}\nabla_{(0,k_0(z))}(f)(x,z)\\{+}\lambda\nabla_{(w,k_1(z))}(f)(x,z){+}\beta x \nabla_{(-1,k_2(z))}(f)(x,z).
\end{multline*}

\begin{proposition}\label{InvPropFPD}
Under Assumptions~A-a and~A-b and if the coordinates of functions $k_0$, $k_1$ and $k_2$ are bounded and all coordinates of $\gamma$ are positive then the Markov process $(X^w(t),Z^w(t))$ has a unique invariant distribution on $\N^{1{+}\ell}$.
\end{proposition}
\begin{proof}
  Since the state space is at most countable, the proof is simpler than its continuous counterpart, Proposition~\ref{InvPropFPD}, where annoying technical intricacies hide the simplicity of the result.  Let $C_k$ be an upper bound for the coordinates of $k_i$, $i{=}0$, $1$, $2$. For $(x,z){\in}\N^{\ell+1}$, define, for $\eta{>}0$, $f(x,z){=}x{+}\eta(z_1{+}\cdots{+}z_\ell)$. We have
\begin{multline*}
Q(f)(x,z) \le {-}x+\lambda w {-}\beta x-\eta \sum_{j=1}^\ell \gamma_j z_j +\eta \ell (1{+}\lambda)C_k+\eta \ell \beta C_k x
\\\le {-}\beta(1{-}\eta\ell C_k)x {-}\min(1,\gamma_j) f(x,z){+}D,
\end{multline*}
with $D$ a constant. If $\eta$ is chosen so that $\eta{<}1/\ell C_k$, then there exists $\gamma{>}0$ and $K$ such that $Q(f)(x,z){<}{-}\gamma$ holds whenever $f(x,z){>}K$. We can now use Proposition~8.14 of~\cite{robert_stochastic_2003} to conclude the proof of the proposition.
\end{proof}
\subsection*{A Discrete Version of Calcium-Based Models}
\label{secsec:calciumqueue}
A comparison between continuous and discrete models of calcium-based STDP is presented in Section~\ref{ap:figures}, and illustrated by Figure~\ref{figS:calciumbased}.  The state of the system corresponds to the case when $(Z(t))$ is a one-dimensional process $(C(t))$ solution of the SDE,
\[
\begin{cases}
\quad  \diff C(t) &= \displaystyle -\mathcal{N}_{I,\gamma C}(\diff t)+ C_1\mathcal{N}_{\lambda}(\diff t)+C_2\mathcal{N}_{I,\beta X}(\diff t),\\
\quad \diff \Omega_{a}(t) &=\displaystyle
        \left(-\alpha\Omega_{a}(t) {+} h_a(C(t))\right)\diff t,\quad a{\in}\{p,d\},\\
\end{cases}
\]
where  $C_1$, $C_2{\in}\N$ and, for $a{\in}\{p,d\}$, $A_a{\in}\N$ and $h_a$ is a non-negative function.

\begin{definition}\label{FVD}
For a fixed $W{=}w$, the Markov process $(X^w(t),C^w(t))$ is defined by its transition rate matrix $Q_C{=}(q_C((x,c),(x',c')))$ is given by, for $(x,c){\in}\N^2$,
\[
\begin{cases}
q_C((x,c),(x{+}w,c{+}C_1)){=}\lambda, \qquad  q_C((x,c),(x{-}1,c)){=}x,\\
q_C((x,c),(x,c{-}1)){=}\gamma c, \qquad q_C((x,c),(x{-}1,c{+}C_2)){=}\beta x.
\end{cases}
\]
\end{definition}

\begin{figure}[ht]
\fontsize{9pt}{8pt}\selectfont
\centering{  \makebox[\textwidth][c]{
\scalebox{0.8}{
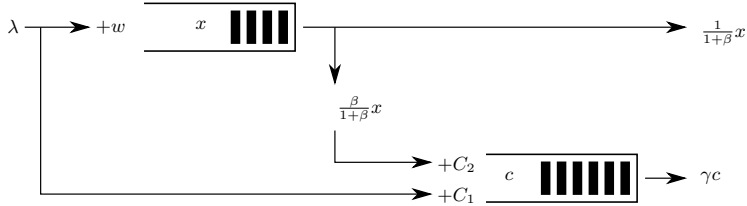}
  }
  }
\caption{Stochastic Queue for the Associated Fast Process of the Discrete Calcium-Based Model}
\label{figS:calciumsq}
\end{figure}

This process can be seen as a network of two $M/M/\infty$ queues with simultaneous arrivals,
see Chapter~6 of~\citet{robert_stochastic_2003}, as illustrated in Figure~\ref{figS:calciumsq}.


\begin{proposition}[Equilibrium of Fast Process] \label{EFPD}
For $w{\in}\N$, the Markov process on $\N^2$ of Definition~\ref{FVD} has a unique invariant distribution $\Pi^{\textup{CQ}}_w$, and the generating function of $C^w$ is given by, for $u{\in}[0,1]$,
\begin{equation}\label{GenC}
    E\left(u^{C^w}\right)=
    \exp\left({-}\lambda\int_{0}^{+\infty}\left(1{-}\Delta(u,s,w)\right)\diff s\right),
\end{equation}
with
\[
\Delta (u,s,w) = \left(\rule{0mm}{4mm}1{+}(u{-}1)p_1(s)\right)^{C_1}\left(1{+}\sum_{i=1}^{C_2}(u{-}1)^k p_2(s,k)\right)^w
\]
\[
p_1(s)=e^{-\gamma s}\text{ and }
  p_2(s,k) = \frac{\beta}{\beta{+}1{-}\gamma k} \binom{C_2}{k}\left(e^{-\gamma k s}{-}e^{-(\beta+1)s}\right).
\]
\end{proposition}
Due to its use in the scaling results of~\citet{robert_stochastic_2020_2}, only the distribution of the calcium variable $C^w$ is considered. The joint generating function of $(X^w,C^w)$ could be obtained with the same approach. 

It might be tempting to try to solve the equilibrium equations for the transition rates of Definition~\ref{FVD}. It does not seem that there is a way to solve them with generating functions methods. The proof below relies in fact on a convenient representation of the Markov process with a Poisson marked point process, it then gives a satisfactory representation of the equilibrium distribution.
\begin{proof}
To each arrival instant $t$ of the Poisson process ${\cal N}_\lambda$ on $\R$ is associated a vector of $\N^{2w+C_1+wC_2}$
$$\underline{u}{=}((x_i,1{\le}i{\le}w),(y_i,1{\le}i{\le}w),(z_{0,j},1{\le}j{\le}C_1),(z_{i,j},1{\le}i{\le}w,1{\le}j{\le}C_2))$$
  We take $(\underline{U}_n){=}((X_{n,i}),(Y_{n,i}),(Z_{n,i,j}))$, where $(X_{n,i})$, $(Y_{n,i})$ and $(Z_{n,i,j})$, sequences of i.i.d. exponentially distributed random variables with respective parameters $1$, $\beta$ and $\gamma$, and independent of ${\cal N}_\lambda$. The interpretation of these variables are as follows, for $1{\le}i{\le}w$, for the $n$th instant of the Poisson process ${\cal N}_\lambda$, 
  \begin{enumerate}
  \item $X_{n,i}$ is the lifetime of the $i$th quantum of potential generated at time $t$ (if any),
  \item $Y_{n,i}$, the duration of time after which this $i$th quantum of potential initiates a firing of the neuron,
  \item $Z_{n,0,j}$, the lifetime of the $j$th quantum of calcium generated at $t$, for $1{\le}j{\le}C_1$,
  \item $Z_{n,i,j}$,  the lifetime of the $i$th quantum of calcium created if the event described by (c) occurs, for $1{\le}j{\le}C_2$.
  \end{enumerate}
Define
  \[
\overline{\cal N}_\lambda(\diff s,\diff \underline{u})\steq{def}\sum_{n{\in}\Z}\delta_{(t_n,\underline{U}_n)},
\]
it is well known that $\overline{\cal N}_\lambda$ is a Poisson marked point process with intensity measure
\begin{equation}\label{IntPoisM}
\mu(\diff s,\diff \underline{u})\steq{def}\lambda\diff s \otimes_{i=1}^w E_{1}(\diff x_i) \otimes_{i=1}^w  E_{\beta}(\diff y_i)\otimes_{j=1}^{C_1} E_{\gamma}(\diff z_{0,j}) \otimes_{i=1}^w\otimes_{j=1}^{C_2} E_{\gamma}(\diff z_{i,j}),
\end{equation}
where $E_\xi(\diff x)$ is the exponential distribution with parameter $\xi{>}0$. See Chapter~5 of \citet{kingman_poisson_1992} for example.

Assuming that $X^w(0){=}C^w(0){=}0$,  with the interpretation of the coordinates of  the mark  $\underline{u}$, it is easy to get the representation, for $t{\ge}0$,
  \[
  X^w(t)=\int_{(0,t]}\sum_{i=1}^w \ind{s+x_i>t,s+y_i>t}\overline{\cal N}_\lambda(\diff s,\diff \underline{u}),
    \]
indeed, if there is an arrival at $s{\le}t$,  its $i$th quantum $i{\in}\{1,\ldots,w\}$ of this arrival with lifetime $x_i$ with firing time $y_i$ is still present at $t$ if $s{+}x_1{>}t$ and  $s{+}y_i{>}t$. Similarly, 
\begin{align*}
  C^w(t) =\int_{(0,t]}\sum_{j=1}^{C_1} \ind{s+z_{0,j}>t}&\overline{\cal N}_\lambda(\diff s,\diff \underline{u})\\
    &+\int_{(0,t]}\sum_{i=1}^w \sum_{j=1}^{C_2} \ind{x_i>y_i,s+y_i<t,s+y_i+z_{i,j}>t}\overline{{\cal N}}_\lambda(\diff s,\diff \underline{u}).
\end{align*}
Using invariance by time-translation of the Poisson process $\overline{{\cal N}}_\lambda$, we get that the random variable $(X^w(t),C^w(t))$ has the same distribution as
\begin{multline*}
\left(\overline{X}^w(t),\overline{C}_w(t)\right)\steq{def}  \left(\int_{(-t,0]}\sum_{i=1}^w \ind{s+x_i>0,s+y_i>0}\overline{\cal N}_\lambda(\diff s,\diff \underline{u})\right.,\\
\int_{(-t,0]} \sum_{j=1}^{C_1}\ind{s+z_{0,j}>0}\overline{\cal N}_\lambda(\diff s,\diff \underline{u})
\left.   {+}\sum_{i=1}^{w} \sum_{j=1}^{C_2} \ind{\substack{x_i>y_i,s+y_i<0,\\s+y_i+z_{i,j}>0}}\overline{{\cal N}}_\lambda(\diff s,\diff \underline{u})\right).
\end{multline*}
The random variables $(\overline{X}^w(t),\overline{C}^w(t))$ are non-decreasing and converging to 
\begin{multline}
\left(\overline{X}^w(\infty),\overline{C}_w(\infty)\right)\steq{def}\left(\int_{(-\infty,0]}\sum_{i=1}^w \ind{s+x_i>0,s+y_i>0}\overline{\cal N}_\lambda(\diff s,\diff \underline{u})\right.,\\
\left.\int_{(-\infty,0]} \left[\sum_{j=1}^{C_1}\ind{s+z_{0,j}>0}+\sum_{i=1}^w \sum_{j=1}^{C_2}\ind{x_i>y_i,s+y_i<0,s+y_i+z_{i,j}>0}\right]\overline{\cal N}_\lambda(\diff s,\diff \underline{u})\right).
\end{multline}
The variable $\overline{X}^w(\infty)$ and $\overline{C}_w(\infty)$ are almost surely finite since, with standard calculations with Poisson processes, we obtain that
\[
\E\left[\overline{X}^w(\infty)\right]=\frac{\lambda}{\beta{+}1}w, \quad
\E\left[\rule{0mm}{4mm}\overline{C}_w(\infty)\right]= \frac{\lambda}{\gamma}\left(C_1{+}
C_2\frac{\beta w}{\beta{+}1}\right).
\]
Recall the formula for Laplace transform of Poisson point processes,
\[
\E\left[\exp\left(\int {-}f(s,\underline{u})\overline{\cal N}_\lambda(\diff s,\diff \underline{u})\right)\right]=
\exp\left(\int \left(1-e^{-f(s,\underline{u})}\right)\mu(\diff s,\diff \underline{u})\right),
\]
for any non-negative Borelian function $f$ on $\R_+^{2w+C_1+wC_2}$, where $\mu$ is defined by Relation~\eqref{IntPoisM}.
See Proposition~1.5 of~\citet{robert_stochastic_2003} for example. For $u{\in}[0,1]$, we therefore get the relation
\begin{multline*}
  {-}\ln\E\left[u^{\overline{C}_w(\infty)}\right]=\\
  \lambda\int_{\R_+}\left(1{-}\E\left[u^{\sum_{j=1}^{C_1}\ind{E_{\gamma,0,j}{>}s}}\right]\E\left[u^{\sum_{i=1}^w\sum_{j=1}^{C_2} \ind{E_{1,i}{>}E_{\beta,i},E_{\beta,i}{<}s{<}E_{\beta,i}{+}E_{\gamma,i,j}}}\right]\right)\diff s=\\
        \lambda\int_{\R_+}\left(1{-}\left(1{-}e^{-\gamma s}{+}ue^{-\gamma s}\right)^{C_1}\E\left[u^{\ind{E_{\beta,1}{<}s{\wedge}E_{1,1}}\sum_{j=1}^{C_2}\ind{s{<}E_{\beta,1}{+}E_{\gamma,1,j}}}\right]^w\right)\diff s,
\end{multline*}
where $(E_{1,i})$, $(E_{\beta,i})$ and $(E_{\gamma,i,j})$ are independent  i.i.d. exponentially distributed random variables with respective parameters $1$, $\beta$ and $\gamma$. We have 
\begin{multline*}
\E\left[u^{\sum_{j=1}^{C_2}\ind{E_{1,1}{>}E_{\beta,1},E_{\beta,1}{<}s{<}E_{\beta,1}{+}E_{\gamma,1,j}}}\right]=\\
1 {-} p(s){+}\E\left[\E\left[1 {-} q(s,E_{\beta,1}) {+} uq(s,E_{\beta,1})\right]^{C_2}\ind{E_{\beta,1}{<}s{\wedge}E_{1,1}}\right]
\end{multline*}
with 
\[
p(s)=\P\left(\rule{0mm}{4mm}E_{\beta,1}{<}E_{1,1}{\wedge}s\right)=\frac{\beta}{\beta+1}\left(1{-}e^{-(\beta+1)s}\right),
\]
and
\[
q(s,E_{\beta,1})=\P\left(\rule{0mm}{4mm}s{-}E_{\beta,1}{<}E_{\gamma,1,1}\middle \vert E_{\beta,1}\right)=e^{-\gamma (s-E_{\beta,1})},
\]
\begin{multline*}
\E\left[\E\left[1 {-} q(s,E_{\beta,1}) {+} uq(s,E_{\beta,1})\right]^{C_2}\ind{E_{1,1}{>}E_{\beta,1},E_{\beta,1}{<}s}\right]=\\
\sum_{k=0}^{C_2}(u{-}1)^k\left[\beta \binom{C_2}{k} e^{-\gamma k s}\int_0^se^{-(\beta+1-\gamma k)h}\diff h\right]=
\sum_{k=0}^{C_2}(u{-}1)^k p_2(s,k)
\end{multline*}
with,
\[
  p_2(s,k) = \frac{\beta}{\beta+1-\gamma k} \binom{C_2}{k}\left(e^{-\gamma k s}{-}e^{-(\beta+1)s}\right)
\]
Note that, $p_2(s,0){=}p(s)$
the proposition is thus proved.
\end{proof}

\printbibliography

\newpage

\appendix

\section{Additional Examples of Plasticity Kernels}\label{AsecMod}
\subsection{\sc Suppression Models}
\label{secsecsec:STDPSup}

Computational models of pair-based  rules of  Section~\ref{secsecsec:pairbased} are easy to implement in large neural networks and they capture some essentials properties of STDP.

Nevertheless, they have been shown to fit poorly with experimental data when more complex protocols are used. See~\citet{froemke_spike-timing-dependent_2002,pfister_triplets_2006}.
For this reason, more detailed models taking into account the influence of several pre- and post-synaptic spikes have been proposed.
\citet{babadi_stability_2016} is a  review of these so-called `triplet-based' rules and their influence on the stability of the synaptic weights distribution.
The model of this section  is a variant of the pair-based model with an additional dependence on earlier instants of post- and pre-synaptic spikes.
Another variant is described in Section~\ref{secsecsec:Triplet}.

It was observed, using triplet-based protocols in~\citet{froemke_spike-timing-dependent_2002}, that preceding pre- and post-synaptic spikes have a `suppression' effect on the Hebbian STDP observed.
Motivated by these experiments, the following model, extending pair-based rules, has been proposed.

If there is a pre-synaptic spike, resp. post-synaptic spike, at time $t{\ge}0$,  we denote by $\ell_1(t)$ [resp. $\ell_2(t)$]  the instant of the last pre-synaptic [resp.  post-synaptic spike], before $t$. For this model, when a pre-synaptic spike occurs at time $t{\ge}0$, the contribution to $\Gamma_a^{\textup{S}}(\cdot,\cdot)(\diff t)$ is the sum over all post-synaptic spikes before time $s{\le}t$ of the quantities
\[
\left(1{-}\Phi_{\textup{S},1} (t{-}\ell_1(t))\right)\left(1{-}\Phi_{\textup{S},2} (s{-}\ell_2(s))\right) \Phi_{a,2}(t{-}s),
\]
and similarly for post-synaptic spikes, where $\Phi_{\textup{S},i}$ is a non-negative non decreasing function verifying $\Phi_{\textup{S},i}(0){\leq}1$ and $\lim_{t\rightarrow+\infty}\Phi_{\textup{S},i}(t){=}0$, for $i{\in}\{1,2\}$.
In particular, if the  instants $t_1$ and $t_2$ of consecutive pre-synaptic spikes  are too close, i.e. $t_2{-}t_1{=}t_2{-}\ell_1(t)$ is small, the synaptic weight is not significantly changed at the instant $t_2$.  And similarly for consecutive post-synaptic spikes.

The plasticity kernels $\Gamma_a^{\textup{S}}$, $a{\in}\{p,d\}$, are defined by, for $m_1$, $m_2{\in}{\cal M}_p(\R_+)$,
\begin{multline*}
  \Gamma_{a}^{\textup{S}}(m_1,m_2)(\diff t) \steq{def}\\
\left[\left(1{-}\Phi_{\textup{S},1}(t_0(m_1,t)\right) \, \int_{(0,t)}\left(1{-}\Phi_{\textup{S},2}( t_0(m_2,s))\right)\Phi_{a,2}(t{-}s)m_2(\diff s)\right]\,m_1(\diff t)\\
  + \left[\left(1{-}\Phi_{\textup{S},2}(t_0(m_2,t)\right)\ \int_{(0,t)}\left(1{-}\Phi_{\textup{S},1}(t_0(m_1,s))\right)\Phi_{a,1}(t{-}s)m_1(\diff s)\right]\,m_2(\diff t)
\end{multline*}
with the $t_0(m,t)$ defined by Equation~\eqref{eq:t0}.

\subsection{\sc Triplet-Based Models}
\label{secsecsec:Triplet}

\citet{pfister_triplets_2006} shows that preceding pre-synaptic spikes enhance the depression obtained for a \emph{post-pre} pairing, whereas preceding post-synaptic spikes lead to a bigger potentiation than in a classical pre-post pairing. The  plasticity kernels $\Gamma_a^{\textup{T}}$, $a{\in}\{p,d\}$ of the associated model are defined by, for $m_1$,  $m_2{\in}{\cal M}_p(\R_+)$,
\begin{multline}
\label{eq:FPaT}
  \Gamma_{a}^{\textup{T}}(m_1,m_2)(\diff t)\steq{def}\\ \left(1{+}\int_{(0,t)}\Phi_{\textup{T},a,1}(t{-}s)m_1(\diff s)\right)\left(\int_{(0,t)}\Phi_{a,2}(t{-}s)m_2(\diff s)\right)\,m_1(\diff t)
\\{+} \left(1{+}\int_{(0,t)}\Phi_{\textup{T},a,2}(t{-}s)m_2(\diff s)\right)\left(\int_{(0,t)}\Phi_{a,1}(t{-}s)m_1(\diff s)\right)\,m_2(\diff t).
\end{multline}
where, for $a{\in}\{p,d\}$, $i{\in}\{1,2\}$, $\Phi_{\textup{T},a,i}$ is
a non-negative non-decreasing function converging to $0$ at infinity. 

It is interesting to note that this model is in contradiction with the suppression model described just before. Both models are based on experimental data from different neuronal cells: visual cortical in~\citet{froemke_spike-timing-dependent_2002}, and hippocampal in~\citet{pfister_triplets_2006}.
A global model taking into account both mechanisms, the \emph{NMDA-model}, is defined in~\citet{babadi_stability_2016}.

\subsection{\sc Voltage-Based Models}
\label{secsecsec:STDPV}

In~\citet{clopath_voltage_2010}, another class of plasticity rules,  voltage-based models, has been used to explain plasticity with biophysical mechanisms, similarly to  calcium-based models.

In particular, filtered traces of the membrane potential $X$ are used in the synaptic update.
Adapting notations from~\citet{clopath_voltage_2010}, we have for depression,

\[
    \Gamma_d(\diff t) = \left[B_d\left(\int_{(0,t)} e^{{-}\gamma_{d,2}(t{-}s)}X(t-s)\diff s{-} \theta_{d}\right)^+\right]{\cal N}_{\lambda}(\diff t),
\]
and for potentiation,
\begin{multline*}
    \Gamma_p(\diff t) = B_p\left(\int_{(0,t)} e^{{-}\gamma_{p,0}(t{-}s)}X(t{-}s)\diff s{-} \theta_{p}\right)^+\\
    {\times}\left(\int_{(0,t)} e^{{-}\gamma_{p,2}(t{-}s)}X(t{-}s)\diff s{-} \theta_{d}\right)^+\\
    {\times}\left(\int_{(0,t)} e^{-\gamma_{p,1}(t-s)}\mathcal{N}_{\lambda}(\diff s)\right)\diff t.
\end{multline*}
See Relations~(1) and~(2) of \citet{clopath_voltage_2010}.

In their model, an adaptive-exponential integrate-and-fire model (AdEx) is used to represent the post-synaptic neuron, instead of a Poisson point process.
They  take $\theta_p$ above the threshold potential of the AdEx model, leading to a simple estimation in terms of the post-synaptic spike train:
\[
    \left(\int_{(0,t)} e^{{-}\gamma_{p,0}(t{-}s)}X(t-s)\diff s{-} \theta_{p}\right)^+\diff t \sim {\cal N}_{\beta, X}(\diff t).
\]
However, $\theta_d$ lies around the resting potential of the neuron, leading to synaptic update that are functions of $X$ directly and not only of the spike-trains.
This feature justifies the denomination {\em voltage-based} models and is not easily taken into account in the framework presented here.
To include such a STDP rule,  one could  extend the definition of a plasticity kernel to $\Gamma(m_1,m_2,x)$ by adding a direct dependence on a \cadlag adapted process $(x(t))$.

We present a variation of the {\em voltage-based} model using filtered functionals of pre- and post-synaptic spike trains that fits in our formalism.
Notice that both models are not equivalent in the sense that in~\citet{clopath_voltage_2010}, sub-threshold-activity can lead to plasticity, whereas our model needs post-synaptic spikes.

If there is a pre-synaptic spike at time $t{>}0$, the synaptic weight is depressed by the quantity
\[
    B_d\left(\int_{(0,t)} e^{{-}\gamma_{d,2}(t{-}s)}\mathcal{N}_{\beta,X}(\diff s){-} \theta_{d}\right)^+,
\]
where, for $x{\in}\R$, $x^+{=}\max(x,0)$,  and if some filtered variable is above some threshold $\theta_d$ at that time.

If there is a pre-synaptic spike at time $t$, the synaptic weight will be potentiated by a quantity involving the product of two filtered variables,
\[
    B_p\left(\int_{(0,t)} e^{{-}\gamma_{p,2}(t{-}s)}\mathcal{N}_{\beta,X}(\diff s){-} \theta_{d}\right)^+\int_{(0,t)} e^{-\gamma_{p,1}(t-s)}\mathcal{N}_{\lambda}(\diff s),
\]

The plasticity kernels are thus defined by, for $m_1$, $m_2{\in}{\cal M}_p(\R_+)$,
\begin{align*}
  \Gamma_{d}^{\textup{V}}\left(m_1,m_2)(\diff t\right) &{\steq{def}}  \left[B_d\left(\int_{(0,t)} e^{{-}\gamma_{d,2}(t{-}s)}m_2(\diff s) {-}\theta_{d}\right)^+\right]\,m_1(\diff t),\\
  \Gamma_{p}^{\textup{V}}\left(m_1,m_2)(\diff t\right) &{\steq{def}}\\
&\hspace{-1cm} \left[B_p\left(\int_{(0,t)} \hspace{-2mm}e^{{-}\gamma_{p,2}(t{-}s)}m_2(\diff s) {-} \theta_{d}\right)^+\hspace{-2mm}\left(\int_{(0,t)} e^{-\gamma_{p,1}(t-s)}m_1(\diff s)\right)\right]\,m_2(\diff t).
\end{align*}

\label{ap:comparison}
\section{Plasticity Models Without Exponential Filtering}
\label{apap:noexp}
In the model of Section~\ref{sec:modelneuralplasticicty}, with Relation~\eqref{eq:Omega} we defined a filtering procedure with an exponential kernel of rate $\alpha{>}0$ for the function $\Omega_a$, where $\Omega_p(t)$ and $\Omega_d(t)$ are used to quantify the past activity of input and output neurons leading to potentiation and depression respectively.
It is given by, for  $a{\in}\{p,d\}$,
\[
\diff \Omega_a(t) = {-}\alpha\Omega_a(t) \diff t  +\Gamma_a({\cal N}_\lambda,{\cal N}_{\beta,X})(\diff t),
\]
where $\Gamma_a({\cal N}_\lambda,{\cal N}_{\beta,X})(\diff t)$ represents the plasticity kernels for potentiation, $a{=}p$, and, for depression, $a{=}d$.

Therefore, the update of the synaptic weight at time $t$ depends on a functional of the synaptic processes that happened \emph{before} $t$.
The dynamic of the synaptic weight $(W(t))$ is defined by,
\[
 \diff W(t) = M\left(\Omega_p(t),\Omega_d(t), W(t)\right)\diff t,
\]
Several studies of computational neuroscience have investigated the role of STDP in a stochastic setting. See~\citet{kempter_hebbian_1999,kistler_modeling_2000,roberts_computational_1999,rubin_equilibrium_2001,morrison_phenomenological_2008} for example.  These references use more ``direct'' dynamics for the synaptic weight.  The update at time $t$  depends only on the current synaptic plastic processes $\Gamma_a({\cal N}_\lambda,{\cal N}_{\beta,X})(\diff t)$ at time $t$, instead of a smoothed version over the past activity. The associated model can be defined so that the corresponding synaptic weight process $(\overline{W}(t))$ satisfies the relation
\[
  \diff \overline{W}(t) = \displaystyle \overline{M}\left(\Gamma_p(\mathcal{N}_\lambda,\mathcal{N}_{\beta,\overline{X}}),\Gamma_d(\mathcal{N}_\lambda,\mathcal{N}_{\beta,\overline{X}}), \overline{W}(t)\right)(\diff t),
    \]
for some functional $\overline{M}$. 
\subsubsection*{Biological Arguments For Exponential Filtering}
It should be noted that the model associated to $(\overline{W}(t))$ does not seem to be in agreement with observations of numerous experimental studies. See~\citet{bi_synaptic_1998,fino_bidirectional_2005,feldman_spike-timing_2012}. In a classical experiment, the protocol to induce plasticity consists in stimulating both neurons at a certain frequency a fixed number of times with a fixed delay $\Delta t$, over a period of up to one or two minutes (60-100 pairings at 1~Hz for example). This part is designed to reproduce conditions of correlations between the two neurons, when mechanisms of plasticity are known to be triggered.  However, measurements of the synaptic weight show that changes take place on a different timescale. After the end of the protocol, it is observed that at least several minutes are necessary to have a significant and stable effect on the synaptic weight. In other words, the change in synaptic weights happens long after the end of the plasticity induction.

For this reason we have chosen to use a filter, possibly with an exponential kernel, on the past synaptic activity.
Therefore it does not only depend on the instantaneous synaptic variable $\Gamma_d({\cal N}_\lambda,{\cal N}_{\beta,X})(\diff t)$ at time $t$, but on the whole past $\Gamma_d({\cal N}_\lambda,{\cal N}_{\beta,X})(\diff s)$, $s{\le}t$, with a smoothing exponential kernel which gives the desired dynamical feature for the synaptic weight.
Another recent article~\citet{robinson_induction} also takes this fact into account by adding an ``ìnduction'' function to the classical models of STDP.

\subsection*{ A Toy Example}
We define 
\[
\begin{cases}
  M(\omega_p,\omega_d,w)=\omega_p{-}\omega_d, &(\omega_p,\omega_d,w){\in}\R_+^2{\times}\R,\\
  \overline{M}(\Gamma_1,\Gamma_2,w)=\Gamma_1{-}\Gamma_2, &\Gamma_1,\Gamma_2{\in}{\cal M}_+(\R_+), \\
  \Gamma_p(\diff t){-}\Gamma_d(\diff t)=\left(F{-}W(t)\right)\diff t,
\end{cases}
\]
with $F{>}0$.
The equations for the time evolution of synaptic weights are given by
\[
    \frac{\diff \overline{W}(t)}{\diff t} = \eps \left(F{-}\overline{W}(t)\right) \text{ and }
  \frac{\diff W(t)}{\diff t} = \alpha^2\int_{0}^t e^{-\alpha(t-s)} (F{-}W(s))\diff s,
\]
with the initial condition $W(0){=}\overline{W}(0){=}w_0{>}0$. We get that
\[
\overline{W}(t)=F{+}(w_0{-}F)e^{-\eps t}, \quad t{\ge}0,
\]
so that $(\overline{W}(t))$ converges to $F$ as $t$ gets large, as it can be expected. By differentiating the relation for $(W(t))$ we obtain,
\[
\frac{\diff^2 W(t)}{\diff t^2} +\alpha \frac{\diff W(t)}{\diff t}+ W(t)= F,
\]
with $W(0){=}w_0$ and $W'(0){=}0$. If we take $\alpha{=}2\eps$ with $\eps{<}1$, we get that
\[
W(t)=F+(w_0{-}F)e^{-\eps t }\left(\cos\left(t\sqrt{1{-}\eps^2}\right)+\sqrt{\frac{\eps^2}{1-\eps^2}}\,\sin\left(t\sqrt{1{-}\eps^2}\right)\right),
\]
in particular $((W(t){-}\overline{W}(t))e^{\eps t}/(w_0{-}F))$ is a periodic function with maximal value of the order of $1/\eps$.
Both functions $(W(t))$ and $(\overline{W}(t))$ converge to $F$ as $t$ goes to infinity at the same exponential rate but differ at the second order.

A comparison of both models is also done in Section~\ref{ap:figures} of the Appendix and illustrated for pair-based rules in Figure~\ref{figS:pairbased} and for calcium-based ones in Figure~\ref{figS:calciumbased}.

\section{Graphical Representation of Models of Plasticity}
\label{ap:figures}
In this section, we will consider several examples of simple dynamics of the Markovian system defined in Section~\ref{sec:markov}.

We will start by comparing the effect of three different Hebbian pair-based rules, both on model with, Section~\ref{sec:markov},  and without,  Section~\ref{apap:noexp}, exponential filtering. Then, we will focus on calcium-based models and show that the discrete model of Section~\ref{ap:stochqueu} can be a good approximation of the continuous model of Section~\ref{secsecsec:cbmclassM}.

We consider two different timescales to compare the induction of plasticity in the model with/without exponential filtering,
\begin{itemize}
\item A fast timescale, on the order of the membrane potential dynamics (see plain black line under each row), where the input and output spike patterns are presented.
\item A slow timescale (20 times slower in this example), on the order of the synaptic weight modifications (see dotted black line), where no input is presented.
\end{itemize}

Input and output spikes patterns are fixed in both Figures (see first row).

\subsection{Pair-based STDP Rules (Figure~\ref{figS:pairbased})}

\begin{figure}[hp!]
\makebox[\textwidth][c]{
\input{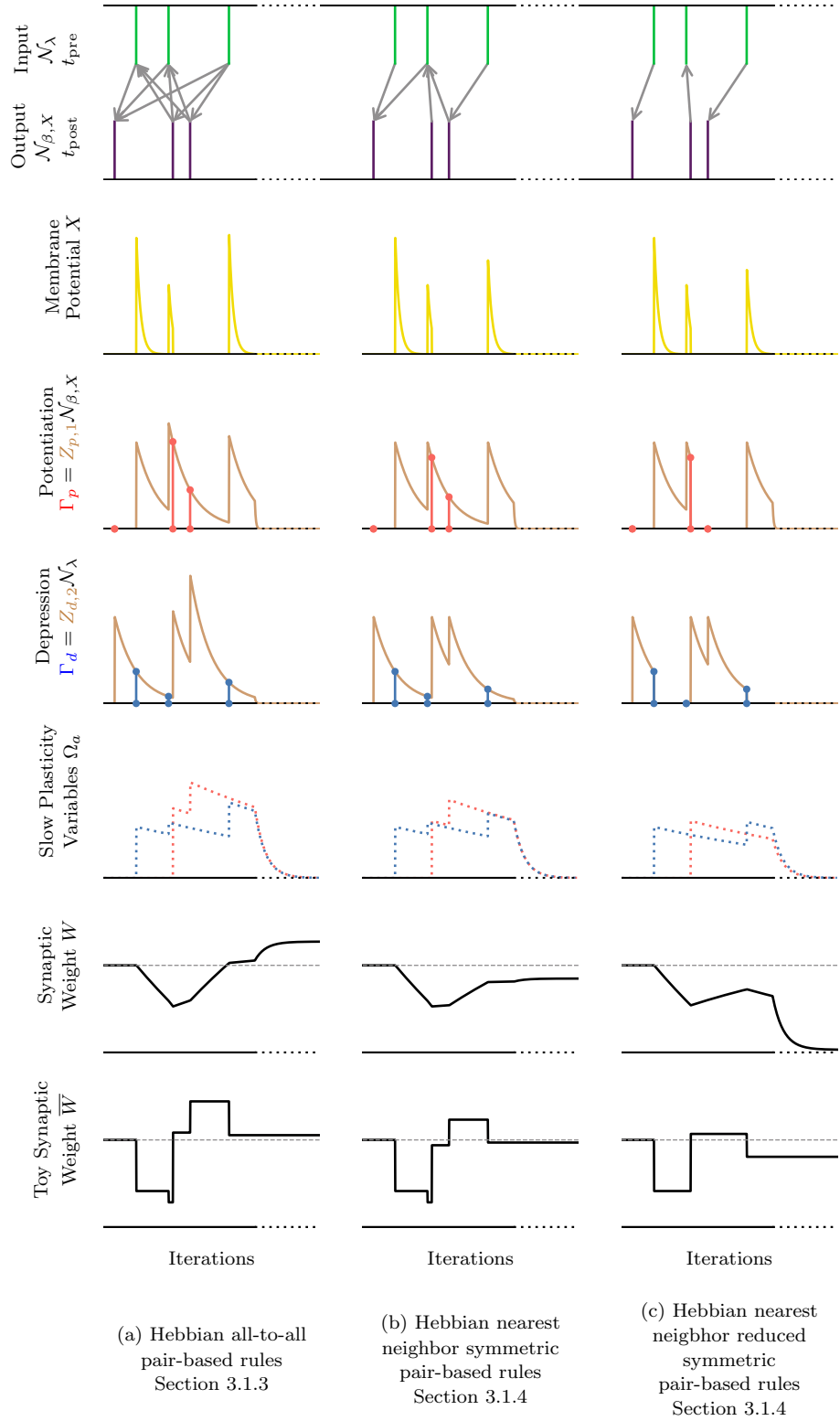}}
\caption{Synaptic Plasticity Kernels for Pair-Based Rules}
\label{figS:pairbased}
\end{figure}

In this section, we describe the dynamics of the different stochastic processes involved in the pair-based STDP model.

In particular, we compare the various interpretation of the pair-based rules that are described in Section~\ref{secsecsec:pairbased} in Figure~\ref{figS:pairbased},
\begin{enumerate}
\item All-to-all Model,
\item Nearest Neighbor Symmetric Model,
\item Nearest Neighbor Reduced Symmetric Model.
\end{enumerate}
The different interactions are represented by grey arrows (first row).

Exponential STDP curves are considered with their associated Markovian description, see Section~\ref{ap:exppbgen}.

Finally, we focus on Hebbian STDP rules with $B_{d,1}{=}0$ and $B_{p,2}{=}0$.

In the second row, the time evolution of the membrane potential,

\[
     \diff X(t) = \displaystyle -X(t)\diff t+W(t{-})\mathcal{N}_\lambda(\diff t)-X(t-)\mathcal{N}_{\beta,X}(\diff t),\\
\]
is presented.
Two interesting facts are to be noted here, at each pre-synaptic spike (green, first row), the current value of the synaptic weight $W(t{-})$ is added to the membrane potential $X(t)$.
It can be seen in this example that the size of the jump varies across time.
In addition,  a complete reset of $X$ occurs after a post-synaptic spike (purple, first row), corresponding to $g(x){=}x$.

Then we focus on the instantaneous synaptic variables $Z_{p,1}$ (brown, third row) and $Z_{d,2}$ (brown, fourth row), that follows different dynamics depending on the rule chosen.
\begin{enumerate}
\item For \emph{all-to-all} pairings, each synaptic spike is paired with all previous post-synaptic spikes, and conversely. They are already described in the main text, by the set of equations, for $a{\in}\{p,d\}$,
    \[
    \begin{cases}
    \quad &\hspace{-3mm}\diff Z_{p,1}(t) ={-}\gamma_{p,1}Z_{p,1}(t)\diff t{+}B_{p,1}\mathcal{N}_{\lambda}(\diff t),\\
    \quad &\hspace{-3mm}\diff Z_{d,2}(t) ={-}\gamma_{d,2}Z_{d,2}(t)\diff t{+}B_{d,2}\mathcal{N}_{\beta,X}(\diff t).
    \end{cases}
    \]
    All pairs of pre-synaptic and post-synaptic spikes are taken into account.
\item For \emph{nearest neighbor symmetric} scheme, each pre-synaptic spike is paired with the last post-synaptic spike, and conversely, the system changes slightly:
    \[
    \begin{cases}
    \quad &\hspace{-3mm}\diff Z_{p,1}(t) ={-}\gamma_{p,1}Z_{p,1}(t)\diff t{+}(B_{p,1}{-}Z_{p,1}(t{-}))\mathcal{N}_{\lambda}(\diff t),\\
    \quad &\hspace{-3mm}\diff Z_{d,2}(t) ={-}\gamma_{d,2}Z_{d,2}(t)\diff t{+}(B_{d,2}{-}Z_{d,2}(t{-}))\mathcal{N}_{\beta,X}(\diff t).
    \end{cases}
    \]
    The variable  $Z_{p,1}$, resp. $Z_{d,2}$ is reset to $B_{p,1}$, resp. $B_{d,2}$, after a pre-synaptic spike, resp. post-synaptic spike.
    \item For \emph{nearest neighbor reduced symmetric} scheme, where only immediate pairing matters, we have:
    \[
    \begin{cases}
    \quad &\hspace{-3mm}\diff Z_{p,1}(t) ={-}\gamma_{p,1}Z_{p,1}(t)\diff t{+}(B_{p,1}{-}Z_{p,1}(t{-}))\mathcal{N}_{\lambda}(\diff t){-}Z_{p,1}(t{-})\mathcal{N}_{\beta,X}(\diff t),\\
    \quad &\hspace{-3mm}\diff Z_{d,2}(t) ={-}\gamma_{d,2}Z_{d,2}(t)\diff t{+}(B_{d,2}{-}Z_{d,2}(t{-}))\mathcal{N}_{\beta,X}(\diff t){-}Z_{d,2}(t{-})\mathcal{N}_{\lambda}(\diff t),
    \end{cases}
    \]
    The variable  $Z_{p,1}$ is reset to $B_{p,1}$, after a pre-synaptic spike and to $0$ after a post-synaptic spike, and conversely for $Z_{d,2}$.
\end{enumerate}
This simple example shows how different pair-based rules shape the instantaneous plasticity variables $Z$.
This dependence is subsequently transferred to the potentiation kernel $\Gamma_p$ (red, third row) and the depression kernel $\Gamma_d$  (blue, fourth row).
With exponential pair-based models, we have $n_{a,0}(z){=}0$, $n_{d,1}(z){=}z_{d,2}$, $n_{p,1}(z){=}0$, $n_{d,2}(z){=}0$, $n_{p,2}(z){=}z_{p,1}$, and therefore, they follow,
\[
    \begin{cases}
    \quad &\hspace{-3mm}\Gamma_p(\diff t) = Z_{p,1}(t-){\cal N}_{\beta,X}(\diff t)\\
    \quad &\hspace{-3mm}\Gamma_d(\diff t) = Z_{d,2}(t-)\mathcal{N}_\lambda(\diff t).
    \end{cases}
\]

It is then not surprising to observe that for a same sequence of pre- and post-synaptic spikes the plasticity kernels are different.

Consequently, it is the same for the slow plasticity variables $\Omega_p$ (red, fifth row) and $\Omega_d$ (blue, fifth row), that follows,
\[
    \begin{cases}
    \quad &\hspace{-3mm}\diff \Omega_p(t) = \displaystyle -\alpha\Omega_p(t)\diff t + Z_{p,1}(t-){\cal N}_{\beta,X}(\diff t)\\
    \quad &\hspace{-3mm}\diff \Omega_d(t) = \displaystyle -\alpha\Omega_d(t)\diff t + Z_{d,2}(t-)\mathcal{N}_\lambda(\diff t),
    \end{cases}
\]

We choose in this example a linear function $M$, leading to the following time evolution of the synaptic weight (sixth row),

\[
    \diff W(t) = \displaystyle \left(\Omega_p(t) - \Omega_d(t)\right)\diff t.
\]
This example shows that a simple change in the STDP rule can lead to very different dynamics for the synaptic weight.
All-to-all rules lead to global potentiation (the dotted line represents the initial value) whereas nearest neighbor rules lead to depression.

Finally, as can be expected from the slow plasticity variables $\Omega_a$ that are still positive long after the end of the stimulus (see in the dotted part), the synaptic weight is modified long after the patterns of spikes.

On the contrary, considering the model without exponential filtering (seventh row),
\[
    \diff \overline{W}(t) = \displaystyle \Gamma_p(\diff t) - \Gamma_d(\diff t),
\]
we see that in that case, the synaptic weight is only updated during the stimulus.
We notice that the polarity of the global plasticity is the same as with exponential filtering, but the dynamics are completely different, as showed with the toy model in Section~\ref{apap:noexp}.

\subsection{Calcium-based STDP Rules (Figure~\ref{figS:calciumbased})}

\begin{figure}[hp!]
\makebox[\textwidth][c]{
\input{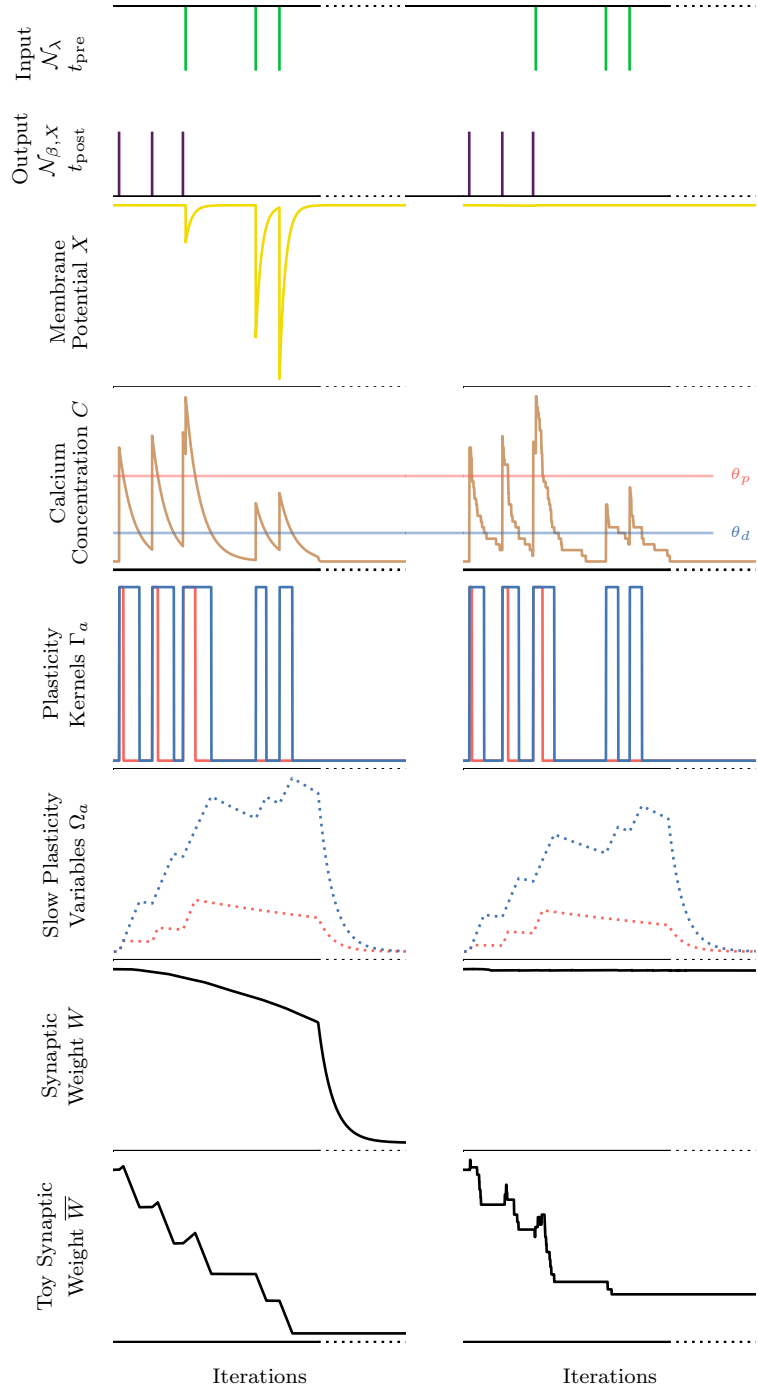}}
\caption{Synaptic Plasticity Kernels for Calcium-Based Rules}
\label{figS:calciumbased}
\end{figure}

In this section, we focus on the dynamics of the calcium-based models,
\begin{enumerate}
\item the continuous version, described in Section~\ref{secsecsec:cbmclassM};
\item the discrete version from Section~\ref{ap:stochqueu}.
\end{enumerate}

The continuous membrane potential (second row, left) follows,
\[
    \diff X(t) = \displaystyle -X(t)\diff t+W(t{-})\mathcal{N}_\lambda(\diff t)-\mathcal{N}_{\beta,X}(\diff t).\\
\]
We consider a different function $g(x){=}1$ than in the previous case.
Its discrete analogue (second row, right) verifies,
\[
    \diff X(t) = \displaystyle{-}\sum_{i=1}^{X(t-)}\mathcal{N}_{1,i}(\diff t) +W(t{-})\mathcal{N}_{\lambda}(\diff t)-\sum_{i=1}^{X(t-)}\mathcal{N}_{\beta,i}(\diff t).\\
\]

It is plainly clear that both processes are almost identical, except that the exponential decay in the continuous model is replaced by a $M/M/\infty$ queue in the discrete case. In the case of large jumps, they lead to a similar dynamical evolution.

The same conclusions can be drawn for the calcium concentration, where the continuous version (third row, left) follows,
\[
    \diff C(t) = {-}\gamma C(t)\diff t{+}C_{1}\mathcal{N}_{\lambda}(\diff t){+}C_{2}\mathcal{N}_{\beta,X}(\diff t),
\]
and the discrete version (third row, right),
\[
    \diff C(t) = \displaystyle -\sum_{i=1}^{C(t-)}\mathcal{N}_{\gamma,i}(\diff t)+ C_1\mathcal{N}_{\lambda}(\diff t)+C_2\sum_{i=1}^{X(t-)}\mathcal{N}_{\beta,i}(\diff t).
\]

In both cases, the plasticity kernels $\Gamma_p$ (fourth row, red) and $\Gamma_d$ (fourth row, blue) verify,

\[
    \begin{cases}
    \quad &\hspace{-3mm}\Gamma_p(\diff t) = \ind{C(t-){\geq}\theta_p}\diff t\\
    \quad &\hspace{-3mm}\Gamma_d(\diff t) = \ind{C(t-){\geq}\theta_d}\diff t.
    \end{cases}
\]

When the calcium reaches the thresholds $\theta_p$ for potentiation (third row, red) and $\theta_d$ (third row, blue), the plasticity kernels are ``activated'' and are equal to $\diff t$.
We see that both models leads to similar values of the kernels, even if some discrepancies start to appear.

The slow plasticity variables (fifth row) are just obtained by integration of the kernels with an exponential filtering,
\[
    \begin{cases}
    \quad &\hspace{-3mm}\diff \Omega_p(t) = \displaystyle -\alpha\Omega_p(t)\diff t + \Gamma_p(\diff t)\diff t\\
    \quad &\hspace{-3mm}\diff \Omega_d(t) = \displaystyle -\alpha\Omega_d(t)\diff t + \Gamma_d(\diff t)\diff t.
    \end{cases}
\]
A second discretization is applied in the synaptic update, the continuous version (sixth row, left) verifies,
\[
    \diff W(t) = \displaystyle \left(A_p\Omega_p(t) - A_d\ind{W(t-){\ge}0}\Omega_d(t)\right)\diff t,
\]
and the discrete one (sixth row, right),
\[
    \diff W(t) =\displaystyle A_p{\cal N}_{\Omega_{p}(t-)}(\diff t)-A_d\ind{W(t-){\ge}A_d}{\cal N}_{\Omega_{d}(t-)}(\diff t).
\]
We note here that we need to force $W$ to stay non-negative in order to have a valid description of the system.
We observe that, even after two different discretizations, both synaptic weights follow a similar evolution.

Using a model without exponential filtering (seventh row) leads to a different dynamical evolution of the synaptic weight, for the continuous model,

\[
    \diff \overline{W}(t) = \displaystyle A_p\Gamma_p(\diff t) - A_d\ind{W(t-){\ge}0}\Gamma_d(\diff t),
\]
and the discrete one,
\[
    \diff \overline{W}(t) =\displaystyle
    A_p\ind{C(t-){\geq}\theta_p}{\cal N}^1_{1}(\diff  t)-A_d\ind{W(t-){\ge}A_d, C(t-){\geq}\theta_d}{\cal N}^2_{1}(\diff t).
\]

As a conclusion, the discrete models approximate well the continuous one and therefore, using the exact expressions of the discrete model can give an interesting insight on the dynamics of the continuous model.

\section{Fast Systems of STDP models}\label{ap:generator}
We first start with the generator of a general STDP of class~${\cal M}$ as in Definition~\ref{def:classM}. 
For $u{=}(x,z,\omega_p,\omega_d,w){\in}\mathcal{S}_{\mathcal{M}}(\ell)$ and $f{\in}\mathcal{C}_b^1\left(\mathcal{S}_{\mathcal{M}}(\ell)\right)$, i.e.~$f$ is a bounded ${\cal C}^1${-}function, and all its respective derivatives are bounded, by using Equations~\eqref{eq:markov}, it is not difficult to show that the extended infinitesimal generator ${\cal A}$ of $(U(t))$ can  be expressed as,
\begin{align*}
  {\cal A}(f)(u)&=
  \left({-}\alpha\omega_p{+}n_{p,0}(z) \right)\frac{\partial f}{\partial \omega_p}(u)
{+}\left({-}\alpha\omega_d{+}n_{d,0}(z) \right)\frac{\partial f}{\partial \omega_d}(u)\\
&  \quad{-}x\frac{\partial f}{\partial x}(u){+}\croc{{-}\gamma {\odot} z {+} k_0,\frac{\partial f}{\partial z}(u)}{+}M(\omega_p,\omega_d,w)\frac{\partial f}{\partial w}(u)\notag\\
& \quad{+}\lambda\left[\rule{0mm}{5mm}f\left(\rule{0mm}{4mm}u{+}we_1{+}k_1(z){\odot} \overline{e}_2{+}n_{p,1}(z)e_{\ell+2}{+}n_{d,1}(z)e_{\ell+3}\right){-}f\left(u\right)\right]\notag\\
& \quad{+}\beta(x)\left[\rule{0mm}{5mm}f\left(\rule{0mm}{4mm}u{-}g(x)e_1{+}k_2(z){\odot} \overline{e}_2{+}n_{p,1}(z)e_{\ell+2}{+}n_{d,1}(z)e_{\ell+3}\right){-}f(u)\right]\notag
\end{align*}
with the following notations, $e_i$ is the unit vectors for the coordinates with index~$i$.
The notation
\[
\left(\frac{\partial f}{\partial z}(u)\right)\steq{def} \left(\frac{\partial f}{\partial u_i}(u), 2{\le}i{\le}\ell{+}1\right)
\]
is for the gradient vector with respect to the coordinates associated to $z$, i.e.~from index $2$ to index $\ell{+}1$.
Finally $\overline{e}_2$ is the vector whose coordinates are $1$ for the indices  associated to $z$ and $0$ elsewhere and, for $a{\in}\R_+^\ell$, the quantity  $a{\odot}\overline{e}_2$ is  the vector whose $i$th coordinate is $a_{i-1}$, for $2{\le}i{\le}\ell{+}1$, and $0$ otherwise.

For sake of completeness, we detail the processes of fast variables for the classical STDP rules described in Section~\ref{secsec:plasticitykernel}.

\subsection{Pair-Based Models with Exponential Kernels $\Phi$}
\label{ap:exppbgen}
For pair-based mechanisms, we follow the classification discussed in~\citet{morrison_phenomenological_2008}:
\begin{itemize}
\item For \emph{all-to-all} pairings,  each synaptic spike is paired with all previous post-synaptic spikes, and conversely. They are already described in the main text, by the set of equations, for $a{\in}\{p,d\}$,
    \[
    \begin{cases}
    \quad &\hspace{-3mm}\diff X(t)  = {-}X(t)\diff t{+}w\mathcal{N}_{\lambda}(\diff t){-}g\left(X(t{-})\right)\mathcal{N}_{\beta,X}\left(\diff t\right),\\
    \quad &\hspace{-3mm}\diff Z_{a,1}(t) ={-}\gamma_{a,1}Z_{a,1}(t)\diff t{+}B_{a,1}\mathcal{N}_{\lambda}(\diff t),\\
    \quad &\hspace{-3mm}\diff Z_{a,2}(t) ={-}\gamma_{a,2}Z_{a,2}(t)\diff t{+}B_{a,2}\mathcal{N}_{\beta,X}(\diff t).
    \end{cases}
    \]
\item In the \emph{nearest neighbor symmetric} scheme each pre-synaptic spike is paired with the last post-synaptic spike, and conversely. The system changes slightly:
    \[
    \begin{cases}
    \quad &\hspace{-3mm}\diff X(t)  = {-}X(t)\diff t{+}w\mathcal{N}_{\lambda}(\diff t){-}g\left(X(t{-})\right)\mathcal{N}_{\beta,X}\left(\diff t\right),\\
    \quad &\hspace{-3mm}\diff Z_{a,1}(t) ={-}\gamma_{a,1}Z_{a,1}(t)\diff t{+}(B_{a,1}{-}Z_{a,1}(t{-}))\mathcal{N}_{\lambda}(\diff t),\\
    \quad &\hspace{-3mm}\diff Z_{a,2}(t) ={-}\gamma_{a,2}Z_{a,2}(t)\diff t{+}(B_{a,2}{-}Z_{a,2}(t{-}))\mathcal{N}_{\beta,X}(\diff t).
    \end{cases}
    \]
    The variable  $Z_{a,1}$, resp. $Z_{a,2}$ is reset to $B_{a,1}$, resp. $B_{a,2}$, after a pre-synaptic spike, resp. post-synaptic spike.
    \item For \emph{nearest neighbor symmetric reduced} scheme, where only immediate pairing matters, we have:
    \[
    \begin{cases}
    \quad &\hspace{-3mm}\diff X(t)  = {-}X(t)\diff t{+}w\mathcal{N}_{\lambda}(\diff t){-}g\left(X(t{-})\right)\mathcal{N}_{\beta,X}\left(\diff t\right),\\
    \quad &\hspace{-3mm}\diff Z_{a,1}(t) ={-}\gamma_{a,1}Z_{a,1}(t)\diff t{+}(B_{a,1}{-}Z_{a,1}(t{-}))\mathcal{N}_{\lambda}(\diff t){-}Z_{a,1}(t{-})\mathcal{N}_{\beta,X}(\diff t),\\
    \quad &\hspace{-3mm}\diff Z_{a,2}(t) ={-}\gamma_{a,2}Z_{a,2}(t)\diff t{+}(B_{a,2}{-}Z_{a,2}(t{-}))\mathcal{N}_{\beta,X}(\diff t){-}Z_{a,2}(t{-})\mathcal{N}_{\lambda}(\diff t),
    \end{cases}
    \]
    \end{itemize}
for exponential pair-based models, with $n_{a,0}(z){=}0$, $n_{a,1}(z){=}z_{a,2}$ and $n_{a,2}(z){=}z_{a,1}$.

\subsection{Nearest Pair-Based Models with General Kernels $\Phi$}
In the case of nearest pair-based models, we have a simple description of the system, based on the time since the last spike as detailed in Section~\ref{STDPNS}.
We define $(Z(t)){=}(Z_{1}(t),Z_{2}(t))$, such that,
\[
\begin{cases}
\diff Z_{1}(t)= \diff t{-}Z_{1}(t{-})\,\mathcal{N}_{\lambda}(\diff t),\\
\diff Z_{2}(t)= \diff t{-}Z_{2}(t{-})\,\mathcal{N}_{\beta,X}(\diff t).
\end{cases}
\]
In this setting, both nearest models are of class ${\cal M}$:
\begin{itemize}
\item The nearest neighbor symmetric model of Relation~\eqref{eq:FPa_nearestPS}, with
  \[
  n_{a,0}(z){=}0,\quad n_{a,1}(z){=} \Phi_{a,2}(z_2), \quad n_{a,2}(z){=} \Phi_{a,1}(z_1).
  \]
 \item The nearest neighbor symmetric reduced model of Relation~\eqref{eq:FPa_nearestPR}, with
  \[
  n_{a,0}(z){=}0,\quad n_{a,1}(z){=} \Phi_{a,2}(z_2)\ind{z_{2}{\le}z_1}, \quad n_{a,2}(z){=} \Phi_{a,1}(z_1)\ind{z_{1}{\le}z_2}.
  \]
\end{itemize}

In fact, we have here two different Markovian systems that represents the same dynamics for nearest exponential STDP rules.

\subsection{Triplet-Based Models}
Generator for triplet-based mechanisms can also be defined in a similar way, see~\citet{babadi_stability_2016} for a list of different implementations.
    \begin{itemize}
    \item The suppression model of Section~\ref{secsecsec:STDPSup} from~\citet{froemke_spike-timing-dependent_2002}, where the Markovian system is given by:
    \[
    \begin{cases}
    \quad &\hspace{-3mm}\diff X(t)  = {-}X(t)\diff t{+}w\mathcal{N}_{\lambda}(\diff t){-}g\left(X(t{-})\right)\mathcal{N}_{\beta,X}\left(\diff t\right),\\
    \quad &\hspace{-3mm}\diff Z_{a,1}(t) ={-}\gamma_{a,1}Z_{a,1}(t)\diff t {+} (1 {-} Z_{s,1}(t{-}))B_{a,1}\mathcal{N}_{\lambda}(\diff t),\\
    \quad &\hspace{-3mm}\diff Z_{a,2}(t) ={-}\gamma_{a,2}Z_{a,2}(t)\diff t {+} (1 {-} Z_{s,2}(t{-}))B_{a,2}\mathcal{N}_{\beta,X}(\diff t),\\
    \quad &\hspace{-3mm}\diff Z_{s,1}(t) ={-}\delta_{1}Z_{s,1}(t)\diff t {+} (1 {-} Z_{s,1}(t{-}))\mathcal{N}_{\lambda}(\diff t),\\
    \quad &\hspace{-3mm}\diff Z_{s,2}(t) ={-}\delta_{2}Z_{s,2}(t)\diff t {+} (1 {-} Z_{s,2}(t{-}))\mathcal{N}_{\beta,X}(\diff t),\\
    \end{cases}
    \]
    with $n_{a,0}(z){=}0$, $n_{a,1}(z){=}(1{-}z_{s,1})z_{a,2}$ and $n_{a,2}(z){=}(1{-}z_{s,2})z_{a,1}$.
    \item The triplet-based model, see~\citet{pfister_triplets_2006}, we have:
    \[
    \begin{cases}
    \quad &\hspace{-3mm}\diff X(t)  = {-}X(t)\diff t{+}w\mathcal{N}_{\lambda}(\diff t){-}g\left(X(t{-})\right)\mathcal{N}_{\beta,X}\left(\diff t\right),\\
    \quad &\hspace{-3mm}\diff Z_{a,1}(t) ={-}\gamma_{a,1}Z_{a,1}(t)\diff t{+}B_{a,1}\mathcal{N}_{\lambda}(\diff t),\\
    \quad &\hspace{-3mm}\diff Z_{a,2}(t) ={-}\gamma_{a,2}Z_{a,2}(t)\diff t{+}B_{a,2}\mathcal{N}_{\beta,X}(\diff t),\\
    \quad &\hspace{-3mm}\diff Z_{s,a,1}(t) = {-}\delta_{a,1}Z_{s,a,1}(t)\diff t{+}D_{a,1}\mathcal{N}_{\lambda}(\diff t),\\
    \quad &\hspace{-3mm}\diff Z_{s,a,2}(t) = {-}\delta_{a,2}Z_{s,a,2}(t)\diff t{+}D_{a,2}\mathcal{N}_{\beta,X}(\diff t),\\
    \end{cases}
    \]
    with $n_{a,0}(z){=}0$, $n_{a,1}(z){=}(1+z_{s,a,1})z_{a,2}$ and $n_{a,2}(z){=}(1+z_{s,a,2})z_{a,1}$.
\end{itemize}

\subsection{Calcium-Based Models}
For models of calcium-based plasticity, we have:
\begin{itemize}
    \item Calcium transients as exponential traces in~\citet{graupner_calcium-based_2012}, which is the dynamics used as an example in this paper.  The system is,
    \[
    \begin{cases}
    \quad \diff X(t) &\hspace{-3mm}= {-}X(t)\diff t{+}w\mathcal{N}_{\lambda}(\diff t){-}g(X(t{-}))\mathcal{N}_{\beta,X}(\diff t), \\
    \quad \diff C(t) &\hspace{-3mm}= {-}\gamma C(t)\diff t{+}C_{1}\mathcal{N}_{\lambda}(\diff t){+}C_{2}\mathcal{N}_{\beta,X}(\diff t).
    \end{cases}
    \]
    \item Calcium transients modeled in a discrete setting as for the example in Section~\ref{ap:stochqueu}.
    The associated Markov process has the following transitions transition rates, for $(x,c){\in}\N^2$,
\[
(x,c)\longrightarrow
\begin{cases}
\hspace{1mm}(x{+}w,c{+}C_1) & \lambda, \\
\hspace{1mm}(x{-}1,c)   & x,
\end{cases}
\hspace{2cm}
\longrightarrow
\begin{cases}
\hspace{1mm}(x,c{-}1)   & \gamma c, \\
\hspace{1mm}(x{-}1,c{+}C_2)  & \beta x.
\end{cases}
\]
\end{itemize}
The functions of calcium-based models are given by, for $a{\in}\{p,d\}$, $n_{a,0}(c){=}h_a(c)$, $n_{a,1}(x,c){=}0$ and  $n_{a,2}(c){=}0$.

\subsection{Voltage-Based Models}
Models of Section~\ref{secsecsec:STDPV}, which are adaptations of~\citet{clopath_voltage_2010} by replacing the direct dependence on filtered traces of $X$, can also be analyzed with this formalism.
The dynamics are given by
    \[
    \begin{cases}
    \quad &\hspace{-3mm}\diff X(t)  = {-}X(t)\diff t{+}w\mathcal{N}_{\lambda}(\diff t){-}g\left(X(t{-})\right)\mathcal{N}_{\beta,X}\left(\diff t\right),\\
    \quad &\hspace{-3mm}\diff Z_{p,1}(t) ={-}\gamma_{p,1}Z_{p,1}(t)\diff t{+}\mathcal{N}_{\lambda}(\diff t),\\
    \quad &\hspace{-3mm}\diff Z_{a,2}(t) ={-}\gamma_{a,2}Z_{a,2}(t)\diff t{+}\mathcal{N}_{\beta,X}(\diff t),
    \end{cases}
    \]
with $n_{a,0}(z){=}n_{p,1}(z){=}n_{d,2}(z){=}0$, $n_{p,2}(z){=}B_{p}z_{p,1}(z_{p,2}{-}\theta_d)^+$, $n_{d,1}(z){=}B_{d}(z_{d,2}{-}\theta_d)^+$.

\end{document}